\documentclass[12pt,reqno]{amsart}
\usepackage{amssymb,amsmath,amsthm,enumerate,verbatim,bbm}
\usepackage[a4paper]{geometry}

\DeclareMathOperator{\Ran}{Ran}
\DeclareMathOperator{\Dom}{Dom}

\DeclareMathOperator{\Tr}{Tr}

\DeclareMathOperator{\Lip}{Lip}

\DeclareMathOperator{\supp}{supp}

\DeclareMathOperator{\esssup}{ess\ sup}

\DeclareMathOperator{\BMO}{BMO}
\DeclareMathOperator{\VMO}{VMO}
\DeclareMathOperator{\CMO}{CMO}
\DeclareMathOperator{\DOI}{DOI}
\DeclareMathOperator{\Smooth}{Smooth}
\newcommand{\sing}{\text{sing}}
\DeclareMathOperator{\const}{const}
\newcommand{\comp}{\text{\rm comp}}

\renewcommand\Im{\text{\rm Im}\,}
\renewcommand\Re{\text{\rm Re}\,}
\newcommand{\abs}[1]{\lvert#1\rvert}
\newcommand{\Abs}[1]{\left\lvert#1\right\rvert}
\newcommand{\norm}[1]{\lVert#1\rVert}
\newcommand{\Norm}[1]{\left\lVert#1\right\rVert}
\newcommand{\jap}[1]{\langle#1\rangle}

\newcommand{\bbT}{{\mathbb T}}
\newcommand{\bbR}{{\mathbb R}}
\newcommand{\bbC}{{\mathbb C}}
\newcommand{\bbN}{{\mathbb N}}
\newcommand{\bbZ}{{\mathbb Z}}

\newcommand{\bbD}{{\mathbb D}}

\newcommand{\wh}{\widehat}

\newcommand{\calH}{{\mathcal H}}
\newcommand{\calK}{{\mathcal K}}

\newcommand{\calM}{\mathcal{M}}

\newcommand{\calB}{\mathcal{B}}

\newcommand{\calR}{\mathcal{R}}

\newcommand{\Sch}{\mathbf{S}}

\numberwithin{equation}{section}
\renewcommand{\[}{\begin{equation}}
\renewcommand{\]}{\end{equation}}

\theoremstyle{plain}
\newtheorem{theorem}{\bf Theorem}[section]
\newtheorem*{theorem*}{Theorem 1.1$'$}
\newtheorem{lemma}[theorem]{\bf Lemma}
\newtheorem{proposition}[theorem]{\bf Proposition}

\theoremstyle{definition}
\newtheorem{definition}[theorem]{\bf Definition}
\newtheorem*{definition*}{\bf Definition}

\theoremstyle{remark}
\newtheorem*{remark*}{\bf Remark}

\newtheorem{example}[theorem]{\bf Example}
\newtheorem*{example*}{\bf Example}

\newcommand{\eps}{\varepsilon}

\newcommand{\loc}{\text{loc}}
\newcommand{\ac}{\text{(ac)}}

\newcommand{\1}{\mathbbm{1}}
\newcommand{\wf}{\widecheck{f}}

\DeclareFontFamily{U}{mathx}{\hyphenchar\font45}
\DeclareFontShape{U}{mathx}{m}{n}{<5> <6> <7> <8> <9> <10>
<10.95> <12> <14.4> <17.28> <20.74> <24.88> mathx10}{}
\DeclareSymbolFont{mathx}{U}{mathx}{m}{n}
\DeclareFontSubstitution{U}{mathx}{m}{n}
\DeclareMathAccent{\widecheck}{0}{mathx}{"71}

\date{15 January 2019}

\title[Kato smoothness and functions of operators]{Kato smoothness and functions of perturbed self-adjoint operators}

\author{Rupert L. Frank}
\address[Rupert L. Frank]{Mathematisches Institut, Ludwig-Maximilans Universit\"at M\"unchen, Theresienstr. 39, 80333 M\"unchen, Germany, and Department of Mathematics, California Institute of Technology, Pasadena, CA 91125, USA}
\email{rlfrank@caltech.edu}

\author{Alexander Pushnitski}
\address[Alexander Pushnitski]{Department of Mathematics, King's College London, Strand, London, WC2R 2LS, UK}
\email{alexander.pushnitski@kcl.ac.uk}

\begin{document}

\begin{abstract}
We consider the difference $f(H_1)-f(H_0)$ for self-adjoint operators $H_0$ and $H_1$ 
acting in a Hilbert space. We establish a new class of estimates for 
the operator norm and the Schatten class norms of this difference. 
Our estimates utilise ideas of scattering theory and involve conditions on $H_0$ and $H_1$
in terms of the Kato smoothness. 
They allow for a much wider class of functions $f$ (including some unbounded ones) 
than previously available results do. 
As an important technical tool, we propose a new notion of Schatten class
valued smoothness and develop a new framework for double operator integrals. 
\end{abstract}

\maketitle

\section{Introduction}\label{sec.a}

\subsection{Setting of the problem}
Let $H_0$ and $H_1$ be self-adjoint operators in a Hilbert space $\calH$, and let $f$ be a complex-valued
function on the real line. 
In the framework of perturbation theory, the problem of estimating the difference 
$$
D(f):=f(H_1)-f(H_0)
$$
either in the operator norm or in a Schatten class norm 
often arises. 
First, to set the scene, we display some known estimates in this context:
\begin{align}
\norm{D(f)}_p&\leq C(p)\norm{f}_{\Lip(\bbR)}\norm{H_1-H_0}_p,
\quad 1<p<\infty;
\label{a2}
\\
\norm{D(f)}_\calB&\leq C\norm{f}_{B_{\infty,1}^1(\bbR)}\norm{H_1-H_0}_\calB,
\label{a1}
\\
\norm{D(f)}_1&\leq C\norm{f}_{B_{\infty,1}^1(\bbR)}\norm{H_1-H_0}_1.
\label{a22}
\end{align}
Here $\norm{\cdot}_p$ is the norm in the standard Schatten class
$\Sch_p$ and $\norm{\cdot}_\calB$ is the operator norm; 
$\Lip(\bbR)$ is the Lipschitz class and $B_{\infty,1}^1(\bbR)$ is a Besov class.  
As usual, the case $p=2$ is very simple (and goes back at least to Birman and Solomyak in 1960s)
and the important special cases $p=1$ and $p=\infty$ (i.e. \eqref{a1} and \eqref{a22}) are exceptional. 
The case $1<p<\infty$ is due to 
Potapov and Sukochev \cite{PotapovSukochev} and the 
cases $p=1$ and $p=\infty$ are due to Peller \cite{Peller1}. 
The estimate \eqref{a2} is obviously sharp ($\Lip$ cannot be replaced by any larger class) 
and the estimates \eqref{a1}, \eqref{a22} are very close to being sharp
(the Besov class $B_{\infty,1}^1(\bbR)$ cannot be replaced by any larger \emph{Besov} class).

In applications (we mainly have in mind the spectral theory of Schr\"odinger operators, see Section~\ref{sec.app})
one often has additional information on the perturbation $H_1-H_0$, which can be expressed in terms
of conditions of the Kato smoothness type. 
In this paper, we propose a framework which allows one to systematically use these smoothness 
conditions in order to improve the estimates on $D(f)$, both in the operator norm and in 
the Schatten class norms.

\subsection{Kato smoothness and the operator norm estimate}

The notion of Kato smoothness was introduced by Kato in his seminal paper 
\cite{Kato1} (with further developments in \cite{Kato2}). In the same paper \cite{Kato1}, 
it was used to prove the existence and completeness of wave operators. We will use this 
notion for a different purpose.

Let $H$ be a self-adjoint operator in a Hilbert space $\calH$
and let $G$ be an operator acting from $\calH$ to another Hilbert space $\calK$. 
We will say  that $G$ is \emph{Kato smooth with respect to $H$} (we will write $G\in \Smooth(H)$), if 
\begin{equation}
\norm{G}_{\Smooth(H)}:=\sup_{\norm{\varphi}_{L^2(\bbR)}=1}\norm{G\varphi(H)}_\calB<\infty.
\label{a3}
\end{equation}
This definition may look unfamiliar, but in fact we show in Section~\ref{sec.b} that it coincides 
with the standard definition of Kato smoothness. 
We will see that the advantage of definition \eqref{a3} is that it extends naturally to Schatten classes. 

We start by stating, somewhat informally, our first (very simple) result; a more precise statement 
will be given in Section~\ref{sec.e}. 

\begin{theorem}\label{thm.a1}
Let $H_0$ and $H_1$ be self-adjoint operators in $\calH$ such that the perturbation $H_1-H_0$ 
factorises as 
$$
H_1-H_0=G_1^*G_0,
$$
with $G_0\in\Smooth(H_0)$ and $G_1\in\Smooth(H_1)$. 
Then for any $f\in\BMO(\bbR)$, one has
\begin{equation}
\norm{D(f)}_\calB\leq2\pi\norm{f}_{\BMO(\bbR)}\norm{G_0}_{\Smooth(H_0)}\norm{G_1}_{\Smooth(H_1)}.
\label{a4}
\end{equation}
\end{theorem}
Here $\BMO(\bbR)$ is the class of functions with \emph{bounded mean oscillation}; 
we recall the description of this class in Section~\ref{sec.d} and fix a suitable norm on it 
(there are many equivalent norms on $\BMO(\bbR)$, but by choosing a specific norm,
we make the constant in front of the right hand side of \eqref{a4} equal to $2\pi$).

Observe that functions in $\BMO(\bbR)$ include some unbounded ones, such as 
$f(\lambda)=\log\abs{\lambda}$. 
This is in sharp contrast with the estimate \eqref{a1}, where $f$ has to be
bounded, continuous and everywhere differentiable (see e.g. \cite{AP} for the differentiability statement).

\subsection{$\Sch_p$-valued smoothness and Schatten norm estimates}
For $0<p<\infty$, let $\Sch_p$ be the standard Schatten class with the (quasi-)norm $\norm{\cdot}_p$
(see Section~\ref{sec.a6} for the definition). 
Generalising \eqref{a3}, we will say that $G\in\Smooth_{p}(H)$  if 
$$
\norm{G}_{\Smooth_{p}(H)}:=\sup_{\norm{\varphi}_{L^2(\bbR)}=1}\norm{G\varphi(H)}_{p}<\infty.
$$
Our main result is the following Schatten class estimate (it will be restated more precisely as
Theorem~\ref{thm.e3} in Section~\ref{sec.e}). 

\begin{theorem}\label{thm.a2}
Let $p,q,r$ be finite positive numbers such that $\frac1p=\frac1q+\frac1r$.
Let $H_1$ and $H_0$ be self-adjoint operators in $\calH$ such that $H_1-H_0=G_1^*G_0$ 
with $G_0\in\Smooth_{q}(H_0)$, $G_1\in\Smooth_{r}(H_1)$.
Then for all $f\in B_{p,p}^{1/p}(\bbR)\cap \BMO(\bbR)$, one has
\[
\norm{D(f)}_p
\leq 
C(p)\norm{f}_{B_{p,p}^{1/p}(\bbR)}
\norm{G_0}_{\Smooth_{q}(H_0)}
\norm{G_1}_{\Smooth_{r}(H_1)}.
\label{a4a}
\]
This extends to $q=\infty$ (resp. to $r=\infty$), if one replaces $\Smooth_q(H_0)$ (resp. $\Smooth_r(H_1)$) 
by $\Smooth(H_0)$ (resp. by $\Smooth(H_1)$). 
\end{theorem}

We recall the definition of the Besov class $B_{p,p}^{1/p}(\bbR)$ in Section~\ref{sec.d}. 
The constant $C(p)$ in \eqref{a4a} depends only on the choice of the functional $\norm{\cdot}_{B_{p,p}^{1/p}(\bbR)}$
in this class. 
We say ``the functional" rather than ``the norm", because technically $\norm{\cdot}_{B_{p,p}^{1/p}(\bbR)}$ is a 
semi-(quasi)norm; \emph{semi} because it vanishes on all polynomials and \emph{quasi} because it satisfies the triangle inequality of the form 
\begin{align*}
\norm{f+g}_{B_{p,p}^{1/p}(\bbR)}&\leq \norm{f}_{B_{p,p}^{1/p}(\bbR)}+\norm{g}_{B_{p,p}^{1/p}(\bbR)}, \quad p\geq1,
\\
\norm{f+g}_{B_{p,p}^{1/p}(\bbR)}^p&\leq \norm{f}_{B_{p,p}^{1/p}(\bbR)}^p+\norm{g}_{B_{p,p}^{1/p}(\bbR)}^p, \quad 0<p<1.
\end{align*}
Requiring that $f\in\BMO(\bbR)$ reduces the arbitrary polynomial in the definition 
of $f$ to an arbitrary additive constant. Observe that for $f=\const$, we have $D(f)=0$. 
  
To illustrate the type of local singularities allowed for functions $f\in B_{p,p}^{1/p}(\bbR)$, 
consider the following example. 
Let $\chi_0\in C_0^\infty(\bbR)$ be a function which equals $1$ in a neighbourhood of the origin
and vanishes outside the interval $(-c,c)$ with some $0<c<1$.
Fix $\alpha\in\bbR$, $a_+,a_-\in\bbC$, and consider the function 
$$
F_\alpha(x)=
\begin{cases}
a_+\chi_0(x)\abs{\log\abs{x}}^{-\alpha}, & x>0,
\\
a_-\chi_0(x)\abs{\log\abs{x}}^{-\alpha}, & x<0.
\end{cases}
$$
\begin{proposition}\label{prp.a1}
Let $F_\alpha$ be as defined above and let $0<p<\infty$. 
\begin{enumerate}[(i)]
\item
If $a_+\not=a_-$, then $f\in B_{p,p}^{1/p}(\bbR)$
if and only if $\alpha>1/p$. 
\item
If $a_+=a_-\not=0$, then $f\in B_{p,p}^{1/p}(\bbR)$  if and only if $\alpha+1>1/p$.
\end{enumerate}
\end{proposition}
In essence, this is an elementary computation using the definition of $B_{p,p}^{1/p}(\bbR)$;
we sketch the proof in the Appendix. 
 
Again, we see that for $p>1$, the functions $F_\alpha\in B_{p,p}^{1/p}(\bbR)$ may be unbounded. 
It is also clear that, in contrast with \eqref{a2}, 
$F_\alpha$ is never in $\Lip$, apart from the trivial cases $a_+=a_-=0$ or $\alpha=0$. 

We note briefly that Theorems~\ref{thm.a1} and \ref{thm.a2} are sharp in the sense that 
the corresponding estimates are saturated for certain operators $H_0$ and $H_1$; 
see Theorem~\ref{thm.sharp} below.

\subsection{Key ideas of the proof}

For a function $f:\bbR\to\bbC$, we denote by $\wf$ the 
divided difference 
\[
\wf(x,y):=\frac{f(x)-f(y)}{x-y}, \quad x,y\in\bbR.
\label{d1}
\]
The Birman-Solomyak formula, which goes back to \cite{BS1} (see also \cite{BS2} for a modern exposition), 
represents the difference $D(f)$
as the \emph{double operator integral} (DOI): 
\[
D(f)=\int_\bbR\int_\bbR  \wf(x,y) dE_{H_1}(x)(H_1-H_0)dE_{H_0}(y), 
\label{BS}
\]
where $E_{H_0}$ (resp. $E_{H_1}$) is the projection-valued spectral measure of $H_0$ (resp. of $H_1$). 
The standard approach (which again goes back to Birman and Solomyak) to the problem of estimating the norm of $D(f)$ 
  is to represent
the map $f\mapsto D(f)$ as a composition of two maps,
\[
f\mapsto \wf \mapsto \int_\bbR\int_\bbR  \wf(x,y) dE_{H_1}(x)(H_1-H_0)dE_{H_0}(y).
\label{fact}
\]
To explain this further, let us
first recall the strategy of the proof of the estimate \eqref{a1}. 
One proves (see \cite{Peller1}) separately the estimates
\begin{align*}
\Norm{\int_\bbR\int_\bbR \wf(x,y) dE_{H_1}(x)(H_1-H_0)dE_{H_0}(y)}
&\leq 
C\norm{\wf}_*\norm{H_1-H_0}_\calB, 
\\
\norm{\wf}_*
&\leq 
C\norm{f}_{B^1_{\infty,1}(\bbR)},
\end{align*}
where $\norm{\cdot}_*$ is a certain norm on the set of integral kernels 
(functions of two variables). 
Putting them together and using the Birman-Solomyak formula, 
this yields \eqref{a1}. 

We use the composition \eqref{fact} as well, but the underlying estimates are different. 
Essentially, we develop an alternative version of the theory of DOI as follows. 
We fix $H_0$, $H_1$ and $G_0\in\Smooth(H_0)$, $G_1\in\Smooth(H_1)$ and consider the map 
$$
a\mapsto \int_\bbR\int_\bbR  a(x,y) dE_{H_1}(x)G_1^*G_0dE_{H_0}(y);
$$
here $a$ is an arbitrary bounded operator in $L^2(\bbR)$ with the integral kernel $a(x,y)$. 
We prove the estimates
\begin{align}
\Norm{\int_\bbR\int_\bbR a(x,y) dE_{H_1}(x)G_1^*G_0dE_{H_0}(y)}
&\leq
\norm{a}_\calB
\norm{G_0}_{\Smooth(H_0)}
\norm{G_1}_{\Smooth(H_1)},
\label{a5a}
\\
\norm{\wf}_\calB
&\leq 
2\pi \norm{f}_{\BMO(\bbR)}. 
\label{a5b}
\end{align}
These estimates, together with the Birman-Solomyak formula, 
yield the proof of Theorem~\ref{thm.a1}. 
Theorem~\ref{thm.a2} is obtained from the Schatten class versions of  \eqref{a5a} and \eqref{a5b}. 

We note that while \eqref{a5a} (and its Schatten class version) is new, the 
estimate \eqref{a5b} is essentially well known. 
In fact, the operator with the integral kernel $\wf(x,y)$ is 
a Hankel operator in disguise; this is well known in 
the Hankel operator community, and  \eqref{a5b} easily follows from there.

\subsection{Some applications}\label{sec.app}
Here we briefly mention some applications of Theorems~\ref{thm.a1} and \ref{thm.a2}; these are developed in detail 
in the forthcoming publication \cite{II}. 
Let 
$$
H_0=-\Delta, \quad H_1=-\Delta+V
\quad\text{ in $L^2(\bbR^d)$, $d\geq1$,}
$$
where the real-valued potential $V$ satisfies the bound
$$
\abs{V(x)}\leq C(1+\abs{x})^{-\rho}, \quad \rho>1.
$$
Under these assumptions, the absolutely continuous spectrum of both $H_0$ and $H_1$ coincides with $[0,\infty)$. 
In applications to mathematical physics (see e.g.  \cite{FP0}), one is often interested in functions $f$
having a cusp-type singularity on the absolutely continuous spectrum and smooth elsewhere. It is also easy to reduce the question 
to functions $f$ compactly supported on $(0,\infty)$. 
\begin{theorem}\cite{II}\label{thm.a4}
\begin{enumerate}[\rm (i)]
\item
Assume $\rho>1$.
Then  any $f\in \BMO(\bbR)$ with compact support in $(0,\infty)$, we have $D(f)\in\calB$. 
\item
Assume $1<\rho\leq d$. 
Then for any $p>\dfrac{d-1}{\rho-1}$ and for any $f\in B_{p,p}^{1/p}(\bbR)$
with compact support in $(0,\infty)$, we have $D(f)\in\Sch_p$. 
\item
Assume $\rho>d$. Then for any $p>d/\rho$ and for any $f\in B_{p,p}^{1/p}(\bbR)$
with compact support in $(0,\infty)$, we have $D(f)\in\Sch_p$. 
\end{enumerate}
\end{theorem}
In the $p=1$ case, this is the result of our previous publication \cite{FP0}. 

In the proof of Theorem~\ref{thm.a4}, the concept of \emph{local} $\Sch_p$-valued smoothness
is important; in other words, one needs inclusions of the type $GE_{H_0}(\Delta)\in \Smooth_p(H_0)$, 
where $\Delta\subset (0,\infty)$. We develop some tools for this in Section~\ref{sec.g7}.

\subsection{The structure of the paper}
In Section~\ref{sec.b} we discuss the classical Kato smoothness and 
in Section~\ref{sec.bb} we introduce and study the $\Sch_p$-valued smoothness. 
In Section~\ref{sec.c} we develop our version of the theory of DOI and prove 
the estimate \eqref{a5a} and its Schatten class version. The key idea of the proof 
is a certain factorisation of the DOI and a subsequent use of interpolation on each factor. 
In Section~\ref{sec.d} we derive the estimate \eqref{a5b} (and its Schatten class version) from the known
estimates for Hankel operators. 
In Section~\ref{sec.e} we put all the components together and prove 
Theorems~\ref{thm.a1} and \ref{thm.a2}. 
Section~\ref{sec.g} contains some additional information. First
we present an example which 
illustrates the sharpness of our main results. Then  
 we consider some extensions: to ``quasicommutators"
$$
f(H_1)J-Jf(H_0)
$$
and to operators of the form
$$
\varphi_1(H_1)^*(f(H_1)-f(H_0))\varphi_0(H_0).
$$
The latter operator is important in applications, which we
develop in a separate paper \cite{II}. 
In Appendix, we sketch the proof of Proposition~\ref{prp.a1} and of another
technical statement of a similar nature.

\subsection{Notation} \label{sec.a6}
Throughout the paper, $\calH$ and $\calK$ are complex separable Hilbert spaces.  
If $H$ is a self-adjoint operator in $\calH$, then 
$E_H(\Lambda)=\1_{\Lambda}(H)$ is the spectral projection of $H$ associated to the set $\Lambda\subset\bbR$.
Here and in what follows $\1_\Lambda$ is the characteristic function of the set $\Lambda$.  
We denote by $\calH^{\ac}(H)$ (resp. $\calH^{(\sing)}(H)$) the absolutely continuous 
(resp. singular) subspace of $H$, and $H^\ac=H|_{\calH^\ac(H)}$.

We will often deal with weakly convergent sequences of bounded operators in a Hilbert space, i.e. 
$(A_nx,y)\to(Ax,y)$ for all elements $x,y$ in the Hilbert space. 
Recall that this is equivalent to $\Tr(A_nB)\to\Tr(AB)$ for all trace class operators $B$. 
Thus, for the sake of uniformity with other types of convergences in function spaces, we shall call this 
$*$-weak convergence in the set of bounded operators. 

The set of bounded operators acting from $\calH$ to $\calK$ is denoted by $\calB(\calH,\calK)$, 
and the corresponding norm is denoted by $\norm{\cdot}_\calB$. 
We use   the class of compact operators $\Sch_\infty(\calH,\calK)$ acting from $\calH$ to $\calK$
and, for $0<p<\infty$, the Schatten class $\Sch_p(\calH,\calK)\subset \Sch_\infty(\calH,\calK)$, defined by 
$$
\norm{A}_p^p=\sum_{n=1}^\infty s_n(A)^p<\infty,
$$
where $\{s_n(A)\}_{n=1}^\infty$ is the sequence of singular values of $A$, enumerated
with multiplicities taken into account. 
Observe that $\norm{\cdot}_p$ is a norm for $p\geq1$, and a quasinorm for $0<p<1$;  the triangle
inequality fails in the latter case.  
However, for $0<p<1$ there is a useful substitute for the triangle inequality due to Rotfeld \cite{Rotfeld} (see also \cite{McCarthy})
\[
\norm{A+B}_p^p
\leq 
\norm{A}_p^p+\norm{B}_p^p, \quad 0<p<1.
\label{b7aa}
\]
We frequently use the  ``H\"older inequality for $\Sch_p$ classes''
$$
\norm{AB}_p\leq \norm{A}_q\norm{B}_r, \quad \tfrac1p=\tfrac1q+\tfrac1r.  
$$

\subsection*{Acknowledgements.} Partial support by U.S. National Science Foundation DMS-1363432 (R.L.F.) is acknowledged.
We are grateful to Barry Simon for discussions related to the proof of Theorem~\ref{thm.b5}. 
A.P. is grateful to Caltech for hospitality.

\section{Kato smoothness}\label{sec.b}

Let $H$ be a self-adjoint operator in $\calH$, and let 
$G:\calH\to\calK$ be an $H$-bounded operator; 
that is, $\Dom H\subset \Dom G$ and the operator 
$GR(z)$ is bounded for all $\Im z\not=0$; 
here and in what follows we denote $R(z)=(H-z)^{-1}$. 

Note that the operator $G$ is not assumed to be closed or closable; 
in fact, in one of our examples $G$ will not admit closure. 
So the stand-alone adjoint $G^*$ is not necessarily well defined, but 
products of the type $(GR(z))^*$ are.

\subsection{Kato smoothness}\label{sec.b1}

We recall (see e.g. \cite[Section~4.3]{Yafaev}) that for an $H$-bounded operator  
$G$,  the following conditions are equivalent:
\[
c_1
:=
(2\pi)^{-2} 
\sup_{\eps>0, \norm{u}=1}
\int_{\bbR}(\norm{GR(x+i\eps)u}_\calK^2+\norm{GR(x-i\eps)u}_\calK^2)dx<\infty;
\label{b1}
\]

\[
c_2
:=
\sup_{\eps>0, \norm{u}=1}\frac{\eps^2}{\pi^2}
\int_{\bbR}(\norm{GR(x+i\eps)R(x-i\eps)u}_\calK^2 dx<\infty;
\label{b2}
\]
\[
c_3
:=
\sup_{(a,b)\subset\bbR}
\frac{{\norm{GE_{H}(a,b)}}_\calB^2}{\abs{b-a}}<\infty.
\label{b3}
\]
If these conditions hold true, then 
$$
c_1=c_2=c_3.
$$
In this case, the operator $G$ is 
called $H$-smooth, and we will write $G\in\Smooth(H)$. 
We will denote 
$$
\norm{G}_{\Smooth(H)}:=\sqrt{c_1}=\sqrt{c_2}=\sqrt{c_3}.
$$
We recall that for $G\in\Smooth(H)$, one has $G|_{\calH^{(\sing)}(H)}=0$; 
here $\calH^{(\sing)}(H)$ is the singular subspace of $H$. 

As mentioned in the Introduction, 
we will need 
a slightly non-standard equivalent definition of smoothness, given by the following theorem. 

\begin{theorem}\label{thm.b1}
$G\in\Smooth(H)$ if and only if
\[
\norm{G\varphi(H)}_\calB\leq C\norm{\varphi}_{L^2}, \quad \forall \varphi\in L^2(\bbR).
\label{b4}
\]
Further, in this case  the norm $\norm{G}_{\Smooth(H)}$ coincides with the optimal 
constant in \eqref{b4}:
$$
\norm{G}_{\Smooth(H)}=\sup_{\norm{\varphi}_{L^2}=1}\norm{G\varphi(H)}.
$$
\end{theorem}

Before proving this theorem, we need to address a minor technical issue: 
since the operator $\varphi(H)$ is in general unbounded,  the definition of 
$G\varphi(H)$ must be made more precise. 
We define $G\varphi(H)$ to be zero on $\calH^{(\sing)}(H)$. 
Next, we will denote by $L^\infty_\comp(H)$ the set of all elements $u\in\calH^{\ac}(H)$ 
for which the function
$$
\frac{d(E_H(-\infty,\lambda)u,u)}{d\lambda}, \quad\lambda\in\bbR
$$
is compactly supported and uniformly bounded on $\bbR$.
It is not difficult to show that $L^\infty_\comp(H)$ is dense in $\calH^{\ac}(H)$. 
It is also easy to see that for $u\in L^\infty_\comp(H)$, 
the element $\varphi(H)u$ is defined for $\varphi\in L^2_\loc(\bbR)$ and 
we have $\varphi(H)u\in\Dom(H)$. 
Thus, $G\varphi(H)u$ is well defined 
for $u\in L^\infty_\comp(H)$. 
Theorem~\ref{thm.b1} says that this definition can be extended to all $u\in\calH^{\ac}(H)$ 
with the norm bound 
\eqref{b4} if and only if $G\in\Smooth(H)$.

\begin{proof}[Proof of Theorem~\ref{thm.b1}] 
Assume that $G\in\Smooth(H)$; let us prove \eqref{b4}. It suffices to consider the 
dense set of functions $\varphi$ of the form
$$
\varphi=\sum_k\varphi_k  \1_{\Lambda_k},
$$
where the sum is finite, $\Lambda_k$ are disjoint intervals in $\bbR$ and $\varphi_k\in\bbC$. 
Then, by \eqref{b3}, 
\begin{multline*}
\norm{G\varphi(H)}_\calB^2
=
\norm{G\varphi(H)(G\varphi(H))^*}_\calB
\leq
\sum_k \abs{\varphi_k}^2\norm{GE_H(\Lambda_k)(GE_H(\Lambda_k))^*}_\calB
\\
\leq
\norm{G}^2_{\Smooth(H)}
\sum_k \abs{\Lambda_k}\abs{\varphi_k}^2
=
\norm{G}^2_{\Smooth(H)}\norm{\varphi}_{L^2}^2,
\end{multline*}
and so we obtain \eqref{b4} with $C=\norm{G}_{\Smooth(H)}$.
The converse follows by taking $\varphi=\1_{(a,b)}$ and by comparing with \eqref{b3}.
\end{proof}

An important ingredient of our construction is 

\begin{theorem}\label{thm.b2}
Let $G\in\Smooth(H)$ and 
let $\{\psi_n\}_{n=1}^\infty$ be an orthonormal sequence in $L^2(\bbR)$. 
Then for any $u\in\calH$: 
\[
\sum_{n=1}^\infty\norm{G\psi_n(H)u}^2_\calK
\leq 
\norm{G}_{\Smooth(H)}^2\norm{u}^2_\calH.
\label{b6}
\]
\end{theorem}
\begin{proof}
Denote by $P_\eps$ the Poisson kernel, 
\[
P_\eps(x)=\frac{\eps}{\pi(x^2+\eps^2)}, \quad x\in\bbR,\quad \eps>0.
\label{b6aa}
\]
For $\eps>0$, let $F_\eps(x)=GP_\eps(H-x)u$, $x\in\bbR$. 
Then by \eqref{b2}, $F_\eps\in L^2(\bbR; \calK)$ with the norm estimate
$$
\int_\bbR\norm{F_\eps(x)}_\calK^2 dx\leq \norm{G}^2_{\Smooth(H)}\norm{u}_\calH^2. 
$$
Let $N\in\bbN$ and let $v_1,\dots,v_N\in\calK$ be any set of elements with 
$\norm{v_n}=1$ for each $n$. Then the set $\{\overline{\psi_n(x)}v_n\}_{n=1}^N$
is orthonormal in the space $L^2(\bbR;\calK)$, and therefore, by the Cauchy-Schwarz
in the same space, 
$$
\sum_{n=1}^N\Abs{\biggl(v_n,\int_\bbR \psi_n(x)F_\eps(x)dx\biggr)_\calK}^2
=
\sum_{n=1}^N\Abs{(\overline{\psi_n}v_n, F_\eps)_{L^2(\bbR;\calK)}}^2
\leq
\int_\bbR\norm{F_\eps(x)}_\calK^2 dx.
$$
Choosing 
$$
v_n=c_n\int_\bbR\psi_n(x) F_\eps(x)dx
$$
with a suitable normalisation constant $c_n$, from here we obtain 
\[
\sum_{n=1}^N
\Norm{\int_\bbR \psi_n(x) F_\eps(x)dx}_\calK^2
\leq 
\int_\bbR\norm{F_\eps(x)}_\calK^2 dx
\leq
\norm{G}^2_{\Smooth(H)}\norm{u}_\calH^2
\label{b6a}
\]
for every $N\in\bbN$.
Next, for every $n\geq1$, we have
$$
\int_\bbR \psi_n(x) F_\eps(x) dx 
=
\int_\bbR \psi_n(x) GP_\eps(H-x)u\, dx
=
G\psi_n^{(\eps)}(H)u,
$$
where 
$$
\psi_n^{(\eps)}(x)=\int_\bbR \psi_n(t)P_\eps(x-t)dt.
$$
Thus, \eqref{b6a} can be written as
$$
\sum_{n=1}^N\norm{G\psi_n^{(\eps)}(H)u}_\calK^2
\leq
\norm{G}^2_{\Smooth(H)}\norm{u}_\calH^2.
$$
Further, by the properties of the Poisson kernel, $\norm{\psi_n^{(\eps)}-\psi_n}_{L^2}\to0$
as $\eps\to0$ for all $n$, and therefore, by Theorem~\ref{thm.b1}, 
$$
\norm{G\psi_n^{(\eps)}(H)u-G\psi_n(H)u}_\calK\to0, \quad \eps\to0
$$
for all $n$. 
It follows that for any $N$, 
$$
\sum_{n=1}^N\norm{G\psi_n(H)u}_\calK^2
\leq
\norm{G}^2_{\Smooth(H)}\norm{u}_\calH^2.
$$
Since $N$ is arbitrary, we obtain \eqref{b6}.
\end{proof}

\subsection{The class $\Smooth_\infty(H)$}

We will write $G\in\Smooth_\infty(H)$, if $G\in \Smooth(H)$ and if 
\[
GE_H(-R,R)\in\Sch_\infty\quad\forall R>0.
\label{b10}
\]
\begin{lemma}\label{lma.b3}
Let $G\in\Smooth_\infty(H)$; then $G\varphi(H)$ is compact for any $\varphi\in L^2(\bbR)$. 
\end{lemma}
\begin{proof}
Since $GE_{H}(-R,R)$ is compact for any $R>0$, the operator $G\varphi(H_0)$ is compact for $\varphi\in L^\infty_{\comp}(\bbR)$. 
Since $L^\infty_{\comp}$ is dense in $L^2$, the bound \eqref{b4} implies that $G\varphi(H_0)$ is compact for all $\varphi\in L^2$, as claimed.
\end{proof}

\subsection{Smoothness with respect to the multiplication operator}\label{sec.b2}
It will be important for us to have a description of the class $\Smooth(\calM)$,
where $\calM$ is the operator of multiplication by the independent variable in a vector-valued 
$L^2$-space. 
Such description was given by Kato in \cite{Kato2}. 
Let $\mathfrak h$ be an auxiliary Hilbert space (which may be finite or infinite dimensional),
and let $\calH=L^2(\bbR;\mathfrak h)$ be the $L^2$ space of $\mathfrak h$-valued functions. 
The operator $\calM$ in $\calH$ is defined as
\[
(\calM f)(x)=xf(x), \quad f\in \Dom \calM,
\label{b6c}
\]
$$
\Dom \calM=\left\{f\in L^2(\bbR;\mathfrak h): \int_{\bbR} \norm{f(x)}_{\mathfrak h}^2(x^2+1)dx<\infty\right\}.
$$

\begin{theorem}\cite{Kato2}\label{thm.b3}
Let $\calM$ be as above and let $G:\calH\to \calK$ be an $\calM$-bounded operator. 
Then $G\in \Smooth(\calM)$ if and only if $G$ can be represented as
\begin{equation}
Gf=\int_\bbR g(x) f(x)dx, \quad
\forall 
f\in \Dom \calM,
\label{e24}
\end{equation}
with some $g\in L^\infty(\bbR;\calB(\mathfrak h,\calK))$. 
Moreover, in this case we have the equality of the norms
\begin{equation}
\norm{G}_{\Smooth(\calM)}=\norm{g}_{L^\infty(\bbR;\calB(\mathfrak h,\calK))}.
\label{e25}
\end{equation}
\end{theorem}
This theorem plays a crucial role in our construction; see Theorem~\ref{thm.b5} and Lemma~\ref{lma.b6} below.  
For this reason and for the sake of completeness we give a proof,  which is essentially a rewording of Kato's proof in \cite{Kato2}. 
\begin{proof}
Let $g\in L^\infty(\bbR;\calB(\mathfrak h,\calK))$ and let  $G$ be defined according to \eqref{e24}.
Then it is clear that for every finite interval $\Lambda$ the operator
$GE_\calM(\Lambda)$ is bounded and 
\begin{equation}
GE_\calM(\Lambda)(GE_\calM(\Lambda))^*=\int_{\Lambda}g(x)g(x)^*dx.
\label{e27}
\end{equation}
It follows that 
$$
\norm{GE_\calM(\Lambda)}_{\calB(\calH,\calK)}^2
=
\Norm{\int_{\Lambda}g(x)g(x)^*dx}_{\calB(\calK)}
\leq 
\abs{\Lambda}\sup_x\norm{g(x)}_{\calB(\mathfrak h,\calK)}^2,
$$
and so, by \eqref{b3}, $G\in\Smooth(\calM)$ and 
\[
\norm{G}_{\Smooth(\calM)}\leq \norm{g}_{L^\infty(\bbR;\calB(\mathfrak h,\calK))}.
\label{e27a}
\]

Conversely, let $G\in \Smooth(\calM)$. 
First we need an auxiliary estimate. 
Observe that $\Dom \calM\subset L^1(\bbR; \mathfrak h)$. 
Write every $f\in \Dom \calM$ as $f=f_1(\calM)f_2$, where
$$
f_1(x)=\norm{f(x)}_{\mathfrak h}^{1/2}
\quad\text{ and }\quad 
f_2(x)=\norm{f(x)}_{\mathfrak h}^{-1/2}f(x) 
\quad \text{ for a.e. $x\in \bbR$.}
$$ 
Then 
$f_1\in L^2(\bbR)$, 
$f_2\in L^2(\bbR;\mathfrak h)$ and 
$$
\norm{f_1}_{L^2(\bbR)}^2=\norm{f_2}_{L^2(\bbR;\mathfrak h)}^2=\norm{f}_{L^1(\bbR;\mathfrak h)}. 
$$
By Theorem~\ref{thm.b1}, we obtain
\begin{multline}
\norm{Gf}
=
\norm{Gf_1(\calM)f_2}
\leq
\norm{G}_{\Smooth(\calM)}\norm{f_1}_{L^2(\bbR)}\norm{f_2}_{L^2(\bbR;\mathfrak h)}
\\
=
\norm{G}_{\Smooth(\calM)}\norm{f}_{L^1(\bbR;\mathfrak h)}.
\label{e26}
\end{multline}
Now
let us establish the existence of $g\in L^\infty(\bbR;\calB(\mathfrak h,\calK))$
that satisfies \eqref{e24}.
In order to define the function $g(x)$, it is easier to start with the adjoint $g(x)^*$. 
Let $\psi\in\calK$; by \eqref{e26}, we have
$$
\abs{(Gf,\psi)}
\leq
\norm{G}_{\Smooth(\calM)}\norm{\psi}_{\calK}\norm{f}_{L^1(\bbR;\mathfrak h)}. 
$$
It follows that the linear functional $f\mapsto (Gf,\psi)$ 
is bounded on $L^1(\bbR;\mathfrak h)$ and therefore 
(see e.g. \cite[Corollary 1.3.22]{HNVW}) 
it can be represented as
\begin{equation}
(Gf,\psi)=\int_\bbR(f(x),g_\psi(x))_{\mathfrak h} dx,
\label{e28}
\end{equation}
with some $g_\psi\in L^\infty(\bbR;\mathfrak h)$ satisfying
\begin{equation}
\norm{g_\psi}_{L^\infty(\bbR;\mathfrak h)}
\leq
\norm{G}_{\Smooth(\calM)}\norm{\psi}_{\calK}. 
\label{e28a}
\end{equation}
By the uniqueness of this representation, $g_\psi$ depends linearly on $\psi$. 
Now for $x\in\bbR$, let us define the operator $g(x)^*:\calK\to\mathfrak h$ by 
$$
g(x)^*\psi:=g_\psi(x)
$$
(to be precise, this should be done on a suitable countable dense set of $\psi$ and a suitable
set of $x$ of full measure -- we omit these details). 
By \eqref{e28a}, we have 
\begin{equation}
\esssup_x \norm{g^*(x)}_{\calB(\calK,\mathfrak h)}\leq \norm{G}_{\Smooth(\calM)}.
\label{e28b}
\end{equation}
Now we can define
$g(x):\mathfrak h\to\calK$ as the adjoint of $g(x)^*$. 
From \eqref{e28} we obtain that 
$$
\int_{\bbR} (g(x)f(x),\psi)_{\calK}dx
=
(Gf,\psi)_\calK
$$
for all  $f\in \Dom(\calM)$. This yields 
\eqref{e24}.
From \eqref{e27a} and \eqref{e28b} we obtain the equality of the norms \eqref{e25}.
\end{proof}

\begin{example}\label{exa.b3}
Let $\mathfrak h=\calK$ and let $g(x)=I$ for all $x$, i.e.,
$$
Gf=\int_\bbR f(x)dx, \quad f\in\Dom \calM.
$$
Then $G\in\Smooth(\calM)$ and $\norm{G}_{\Smooth(\calM)}=1$. 
It is easy to see that $G$ is not closable. 
\end{example}

\section{$\Sch_p$-valued smoothness}\label{sec.bb}

\subsection{Definition and characterisation}\label{sec.b3}

\begin{definition}\label{def.b4}
For $0<p<\infty$, we write $G\in\Smooth_{p}(H)$, if $G\in\Smooth(H)$ and if
for some $C>0$ and for all $\varphi\in L^2(\bbR)$, 
$$
\norm{G\varphi(H)}_{p}\leq C\norm{\varphi}_{L^2}.
$$
In this case we set
$$
\norm{G}_{\Smooth_{p}(H)}=\sup_{\norm{\varphi}_{L^2}=1}\norm{G\varphi(H)}_{p}.
$$
\end{definition}
\begin{lemma}\label{lma.bb1}
Let $p\geq2$; then 
\[
\norm{G}_{\Smooth_{p}(H)}=\sup_{\Lambda\subset \bbR}\abs{\Lambda}^{-1/2}\norm{GE_H(\Lambda)}_{p},
\quad p\geq2,
\label{b11}
\]
where the supremum is taken over all finite intervals $\Lambda$. 
\end{lemma}
\begin{proof}
Denote by $A$ the right hand side of \eqref{b11}. 
The inequality $\norm{G}_{\Smooth_p(H)}\geq A$
follows by taking $\varphi=\1_\Lambda$. The converse inequality follows by the same 
calculation as in the proof of Theorem~\ref{thm.b1}, with Schatten norms instead of the operator norms. 
Indeed, for 
$$
\varphi=\sum_k\varphi_k\1_{\Lambda_k}, 
$$
we have 
\begin{multline*}
\norm{G\varphi(H)}_{\Sch_p}^2
=
\norm{G\varphi(H)(G\varphi(H))^*}_{\Sch_{p/2}}
\leq
\sum_k\abs{\varphi_k}^2\norm{GE_H(\Lambda_k)(GE_H(\Lambda_k))^*}_{\Sch_{p/2}}
\\
=
\sum_k\abs{\varphi_k}^2\norm{GE_H(\Lambda_k)}_{\Sch_{p}}^2
\leq
A\sum_k\abs{\varphi_k}^2\abs{\Lambda_k}
=
A\norm{\varphi}_{L^2}^2,
\end{multline*}
which gives the required bound. 
\end{proof}

For $0<p<2$ the argument of Lemma~\ref{lma.bb1} is no longer valid, as the triangle inequality fails for 
the quasi-norm $\norm{\cdot}_{p/2}$.

For $p\geq2$, $\Sch_{p}$-valued smoothness with respect to the multiplication operator is easy to characterise.
For $0<p<2$, we have only a necessary condition for $\Sch_{p}$-valued smoothness. 

\begin{theorem}\label{thm.b5}
Let $\calM$ be the multiplication operator \eqref{b6c} in $\calH=L^2(\bbR;\mathfrak h)$
and let $G:\calH\to \calK$ be an $\calM$-bounded operator. 
\begin{enumerate}[\rm (i)]
\item
Let $p\geq2$; 
then $G\in \Smooth_{p}(\calM)$  if and only if $G$ can be represented as in \eqref{e24}
with some $g\in L^\infty(\bbR;\Sch_{p}(\mathfrak h,\calK))$. 
Moreover, in this case we have the equality of the norms
$$
\norm{g}_{L^\infty(\bbR;\Sch_{p}(\mathfrak h,\calK))}
=
\norm{G}_{\Smooth_{p}(\calM)}, \quad p\geq2.
$$
\item
Let $0<p<2$. 
If $G\in \Smooth_{p}(\calM)$, then $G$ can be represented as in \eqref{e24}
with some $g\in L^\infty(\bbR;\Sch_{p}(\mathfrak h,\calK))$ and
\begin{equation}
\norm{g}_{L^\infty(\bbR;\Sch_{p}(\mathfrak h,\calK))}
\leq
\norm{G}_{\Smooth_{p}(\calM)}, \quad 0<p<2.
\label{b12b}
\end{equation}
\end{enumerate}
\end{theorem}

After the proof of this theorem we will give an example that shows that for $0<p<2$ an operator $G$ represented as in \eqref{e24} with some $g\in L^\infty(\bbR;\Sch_{p}(\mathfrak h,\calK))$ does not necessarily belong to $\Smooth_{p}(\calM)$, so one cannot expect equality in \eqref{b12b}.

We need the following well-known lemma:
\begin{lemma}\label{traceconv}
Let $\{A_n\}_{n=1}^\infty$ be a sequence of non-negative operators 
which converges $*$-weakly to an operator $A$. Then
$$
\Tr A \leq \liminf_{n\to\infty} \Tr A_n
$$
(with the understanding that the left side is finite if the right side is).
\end{lemma}

\begin{proof}
Let $\{e_j\}_{j=1}^\infty$ be an orthonormal basis of the underlying Hilbert space. 
Then for any $J\in\bbN$,
$$
\sum_{j=1}^J (Ae_j,e_j) 
= 
\liminf_{n\to\infty} \sum_{j=1}^J (A_n e_j,e_j)
\leq
\liminf_{n\to\infty} \Tr A_n. 
$$
The assertion follows as $J\to\infty$ by monotone convergence.
\end{proof}

\begin{proof}[Proof of Theorem \ref{thm.b5}]
Let $g\in L^\infty(\bbR;\Sch_{p}(\mathfrak h,\calK))$ for some $p\geq2$ and let $G$ be defined according to \eqref{e24}.
Then, using  \eqref{e27}, we obtain
\begin{multline*}
\norm{GE_{\calM}(\Lambda)}_{p}^2
=
\norm{GE_{\calM}(\Lambda)(GE_{\calM}(\Lambda))^*}_{p/2}
=
\Norm{\int_{\Lambda}g(x)g(x)^*dx}_{p/2}
\\
\leq
\norm{gg^*}_{L^\infty(\bbR;\Sch_{p/2}(\mathfrak h,\calK))}\abs{\Lambda}
=
\norm{g}^2_{L^\infty(\bbR;\Sch_{p}(\mathfrak h,\calK))}\abs{\Lambda}.
\end{multline*}
By Lemma~\ref{lma.bb1}, 
it follows that $G\in\Smooth_{p}(\calM)$ and 
$$
\norm{G}_{\Smooth_{p}(\calM)}
\leq
\norm{g}_{L^\infty(\bbR;\Sch_{p}(\mathfrak h,\calK))}.
$$

We now prove the converse implication and assume that $G\in\Smooth_{p}(\calM)$ for some $p>0$. 
Since $\Smooth_{p}(\calM)\subset\Smooth(\calM)$, 
by Theorem~\ref{thm.b3} we have the representation \eqref{e24} with some $g\in L^\infty(\bbR;\calB(\mathfrak h,\calK))$. 
We claim that 
$$
\Norm{\frac1{2\eps}\int_{\lambda-\eps}^{\lambda+\eps} g(x)g(x)^*dx- g(\lambda)g(\lambda)^*}_{\calB(\calK)}\to0
\text{  as $\eps\to0$ for a.e. $\lambda\in\bbR$.}
$$
This is the Lebesgue differentiation theorem for functions on $\bbR$ valued in the Banach space $\calB(\calK)$; 
see e.g.  \cite[Theorem 2.3.4]{HNVW}. 
Since the function $t\mapsto t^{p/2}$ is continuous on $[0,\infty)$, we infer that
$$
\Norm{\left( \frac1{2\eps}\int_{\lambda-\eps}^{\lambda+\eps} g(x)g(x)^*\,dx \right)^{p/2} - \left( g(\lambda)g(\lambda)^* \right)^{p/2}}
\to0
\qquad\text{ for a.e.}\ \lambda\in\bbR \,.
$$
By the lower semi-continuity of the trace which we have recalled in Lemma \ref{traceconv}, we obtain for almost every $\lambda\in\bbR$
\begin{multline*}
\|g(\lambda)\|_{p}^{p} 
= 
\Tr \left( g(\lambda)g(\lambda)^*\right)^{p/2} 
\leq 
\liminf_{\epsilon\to 0} \Tr \left( \frac1{2\eps}\int_{\lambda-\eps}^{\lambda+\eps} g(x)g(x)^*dx \right)^{p/2} 
\\
=
\liminf_{\epsilon\to 0} \left\| \frac1{2\eps}
GE_{\mathcal M}(\lambda-\eps,\lambda+\eps)
(GE_{\mathcal M}(\lambda-\eps,\lambda+\eps))^* \right\|_{p/2}^{p/2} \,.
\end{multline*}
By the definition of smoothness with $\varphi=\frac1{\sqrt{2\eps}}\1_{(\lambda-\eps,\lambda+\eps)}$ we have
$$
\left\| \frac1{2\eps}
GE_{\mathcal M}(\lambda-\eps,\lambda+\eps)
(GE_{\mathcal M}(\lambda-\eps,\lambda+\eps))^* \right\|_{p/2}^{p/2}
\leq \norm{G}^2_{\Smooth_{p}(\calM)} \,.
$$
This implies $g\in L^\infty(\bbR;\Sch_{p})$ and 
$$
\norm{g}_{L^\infty(\bbR;\Sch_{p})}
\leq
\norm{G}_{\Smooth_{p}(\calM)} \,.
$$
This completes the proof of the theorem.
\end{proof}

\begin{example}
Let $\mathfrak h=\calK=\ell^2$, and 
let $(e_n)_{n\in\bbN}$ be the standard basis in $\ell^2$. Define
$$
g(x) = \sum_{n=1}^\infty \1_{(n-1,n)}(x) (\cdot,e_n) e_n, \quad x\in\bbR.
$$
Then clearly $g\in L^\infty(\bbR;\Sch_{p}(\ell^2))$ with $\norm{g}_{L^\infty(\bbR;\Sch_{p})}=1$ for any $p>0$. Moreover, for the interval $\Lambda_N=(0,N)$ we find similarly to the proof of Theorem~\ref{thm.b5}
$$
\norm{GE_{\calM}(\Lambda_N)}_{p}^2
=
\Norm{\int_{\Lambda_N}g(x)g(x)^*dx}_{p/2}
= \Norm{ \sum_{n=1}^N (\cdot,e_n) e_n }_{p/2}
= N^{2/p}.
$$
For $0<p<2$ we conclude that
$$
\sup_{\Lambda\subset\bbR} |\Lambda|^{-1/2} \norm{GE_{\calM}(\Lambda_N)}_{p} 
\geq 
\sup_N \abs{\Lambda_N}^{-1/2}N^{1/p}
= \infty
$$
and therefore $G\notin\Smooth_{p}(\calM)$.
\end{example}

\subsection{An interpolation result}\label{sec.b4}

\begin{lemma}\label{lma.b6}
Let $2<q<\infty$, and let $G\in\Smooth_{q}(H)$.
Then there exists a family of operators $G(z):\calH\to\calK$,
$0\leq\Re z\leq 1$, such that:
\begin{enumerate}[\rm (i)]
\item
$G(z)\in\Smooth(H)$ for all $z$, with $\sup_{0\leq\Re z\leq1}\norm{G(z)}_{\Smooth(H)}<\infty$;
\item
$\norm{G(z)}_{\Smooth(H)}\leq 1$ for $\Re z=0$;
\item
$\norm{G(z)}_{\Smooth_2(H)}^2\leq \norm{G}_{\Smooth_{q}(H)}^{q}$ for $\Re z=1$;
\item
$G(2/q)=G$;
\item
for any $\varphi\in L^2(\bbR)$, the family of bounded operators $G(z)\varphi(H)$ 
is analytic in $z$ for $0<\Re z<1$ and continuous in $z$ for $0\leq \Re z\leq 1$.
\end{enumerate}
\end{lemma}

Before coming to the proof, we recall the following consequence of the spectral theorem for self-adjoint operators. 
Let $H$ be a self-adjoint operator in $\calH$; 
then there exists a Hilbert space $\mathfrak h$ and a linear isometry (not necessarily onto)
\[
U: \calH^{\ac}(H)\to L^2(\bbR; \mathfrak h), \quad \text{ such that } \quad \varphi(H^{\ac})=U^*\varphi(\calM) U,
\label{b7}
\]
for any Borel function $\varphi$ on $\bbR$. Here 
$\calM$ is the multiplication operator \eqref{b6c} in $L^2(\bbR; \mathfrak h)$. 
Further, it is easy to see that $G\in\Smooth(H)$ if and only if  $GU^*\in\Smooth(\calM)$, with 
$$
\norm{G}_{\Smooth(H)}=\norm{GU^*}_{\Smooth(\calM)},
$$
and the same is true for the $\Smooth_q$ norms.

\begin{proof}[Proof of Lemma~\ref{lma.b6}]
By the above remarks, the question is reduced to the case $H=\calM$. 
By Theorem~\ref{thm.b5}, $G$ has the representation 
$$
Gf=\int_\bbR g(x)f(x)dx
$$
with $g\in L^\infty(\bbR; \Sch_{q}(\mathfrak h,\calK))$. 
Write the polar decomposition of $g(x)$ as 
$$
g(x)=\omega(x) \abs{g(x)}, \quad x\in\bbR,
$$
where $\omega(x)$ is a partial isometry for a.e. $x\in\bbR$. 
Now let us define
$$
G(z)f=\int_\bbR g_z(x)f(x)dx, 
\quad
g_z(x)=\omega(x)\abs{g(x)}^{qz/2}.
$$
We have
\begin{itemize}
\item
$g_z\in L^\infty(\bbR;\calB)$ for all $z$, $0\leq \Re z\leq1$, and $\sup_{0\leq\Re z\leq1}\norm{g_z}_{L^\infty(\bbR,\calB)}<\infty$;
\item
$\norm{g_z}_{L^\infty(\bbR;\calB)}\leq1$
for $\Re z=0$;
\item
$\norm{g_z}_{L^\infty(\bbR;\Sch_2)}^2=\norm{g}_{L^\infty(\bbR;\Sch_{q})}^{q}$
for $\Re z=1$;
\item
$g_{2/q}=g$.
\end{itemize}
From here, again using Theorem~\ref{thm.b5}, we obtain the properties (i)--(iii) of $G(z)$. 
The property (iv) is obvious from the definition,
and the property (v) is straightforward to check.
\end{proof}

\section{Double operator integrals}\label{sec.c}

\subsection{Overview}
The notion of double operator integrals (DOI) was initially introduced by Daletskii and Krein
in \cite{DK} and developed by Birman and Solomyak  in \cite{BS1} (see \cite{BS2} for a modern
account of the theory and for further historical references). 
Here we consider DOI from a different viewpoint; essentially, we construct an alternative
version of the theory of DOI under a different set of assumptions.

Throughout this section, $H_0$ and $H_1$ are self-adjoint operators in $\calH$ 
and $G_0$, $G_1$ are operators from $\calH$ to $\calK$
such that
$G_0\in\Smooth(H_0)$ and $G_1\in\Smooth(H_1)$. 
We will work with bounded operators $a$ on $L^2(\bbR)$ 
and with their integral kernels $a(x,y)$. 
(In practice, we will only need the notion of an integral kernel for 
finite rank operators $a$; in this case this notion can be unambiguously defined 
without difficulty.)
Informally speaking, we would like to define the double operator integral
\[
\DOI(a)=\int_\bbR\int_\bbR a(x,y) dE_{H_1}(x)G_1^*G_0dE_{H_0}(y),
\label{c1}
\]
initially for finite rank operators $a$ and eventually for all bounded operators $a$ on $L^2(\bbR)$. 
In other words, for fixed $G_0$, $G_1$, $H_0$, $H_1$, we consider the map 
$$
\DOI: \calB(L^2(\bbR))\to \calB(\calH),
$$
defined initially on the set of all finite rank operators $a$. 
We prove that this map can be extended in a natural way to the whole space 
$\calB(L^2(\bbR))$, that it is bounded and satisfies the 
operator norm and the Schatten norm bounds
\begin{align}
\norm{\DOI(a)}_\calB
&\leq 
\norm{G_0}_{\Smooth(H_0)}\norm{G_1}_{\Smooth(H_1)}\norm{a}_\calB,
\label{c2}
\\
\norm{\DOI(a)}_{p}
&\leq
\norm{G_0}_{\Smooth_{q}(H_0)}\norm{G_1}_{\Smooth_{r}(H_1)}
\norm{a}_p, \quad \tfrac1p=\tfrac1q+\tfrac1r.
\label{c3}
\end{align}

In order to make sense of the integral \eqref{c1},
in the standard approach to the theory of double operator integrals \cite{BS1,BS2}
one has to assume some degree of regularity of the kernel $a(x,y)$. 
In our framework, the regularity of $a(x,y)$ is not needed, as we are
using the smoothness of $G_0$ and $G_1$ instead.

Recall that if $G\in\Smooth(H)$, then $G|_{\calH^{(\sing)}(H)}=0$. 
Thus, it is natural to define $\DOI(a)$ such that it satisfies the property
\[
(\DOI(a)u,v)=0\quad\text{ if $u\in \calH^{(\sing)}(H_0)$ or $v\in\calH^{(\sing)}(H_1)$}
\label{c1a}
\]
(or both). 
Thus, essentially $\DOI(a)$ acts from $\calH^{\ac}(H_0)$ to $\calH^{\ac}(H_1)$.

It will be convenient to use the following notation for the constants in the estimates 
\eqref{c2} and \eqref{c3}: 
\[
A:=\norm{G_0}_{\Smooth(H_0)}\norm{G_1}_{\Smooth(H_1)},
\quad
A_{q,r}:=\norm{G_0}_{\Smooth_{q}(H_0)}\norm{G_1}_{\Smooth_{r}(H_1)}.
\label{c3a}
\]

\subsection{$\DOI(a)$ for finite rank $a$}
We begin by defining $\DOI(a)$ for finite rank operators $a$. 
Let $a$ be given by its Schmidt series, 
\[
a=\sum_{n=1}^N s_n (\cdot,\varphi_n)\psi_n,
\label{c3b}
\]
where $N$ is finite, $\{s_n\}$ are the singular values of $a$ and $\{\varphi_n\}$, $\{\psi_n\}$ are
orthonormal sets. Then the integral kernel of $a$ is given by 
$$
a(x,y)=\sum_{n=1}^N s_n \psi_n(x)\overline{\varphi_n(y)}, \quad x,y\in\bbR.
$$
In this case, we set
\[
\DOI(a)=\sum_{n=1}^N s_n (G_1\psi_n(H_1)^*)^*G_0\varphi_n(H_0)^*.
\label{c5}
\]
From this definition it follows, in particular, that the property \eqref{c1a} is satisfied.

First we need to check that definition \eqref{c5} is independent of the choice of the Schmidt series representation
\eqref{c3b}. This will follow from the next lemma.
\begin{lemma}\label{lma.c1}
For $j=0,1$, let $U_j$ be a diagonalization isometry as in \eqref{b7}, i.e.
$$
U_j: \calH^{\ac}(H_j)\to L^2(\bbR; \mathfrak h), \quad H_j=U_j^*\calM U_j,
$$
where $\mathfrak h$ is a Hilbert space and $\calM$ is the operator of multiplication by the independent variable in $L^2(\bbR;\mathfrak h)$. 
For $v_j\in\calH^\ac(H_j)$, denote $\wh v_j=U_jv_j\in L^2(\bbR;\mathfrak h)$ and write the 
representation of Theorem~\ref{thm.b3} for $G_jU_j^*$ as
$$
G_jU_j^* \wh v_j=\int_\bbR g_j(x) \wh v_j(x)dx, \quad g_j \in L^\infty(\bbR;\calB(\mathfrak h,\calK)),
\quad
j=0,1.
$$
Then for all finite rank operators $a$, we have
\[
(\DOI(a)v_0,v_1)
=
\int_\bbR\int_\bbR
a(x,y)
\bigl(g_0(y)\wh v_0(y),g_1(x)\wh v_1(x)\bigr)_\calK dx\, dy.
\label{c6}
\]
\end{lemma}
\begin{proof}
By linearity, it suffices to prove \eqref{c6} for rank one operators $a$. 
Let $a(x,y)=\psi(x)\overline{\varphi(y)}$.
Then 
\begin{multline*}
(\DOI(a)v_0,v_1)
=
(G_0\varphi(H_0)^*v_0,G_1\psi(H_1)^*v_1)_\calK
\\
=
(G_0U_0^*\varphi(\calM)^*U_0v_0,G_1U_1^*\psi(\calM)^*U_1v_1)_\calK
\\
=
\biggl(
\int_\bbR g_0(y)\overline{\varphi(y)}\wh v_0(y)dy,
\int_\bbR g_1(x)\overline{\psi(x)}\wh v_1(x)dx
\biggr)_\calK
\\
=
\int_\bbR\int_\bbR\psi(x)\overline{\varphi(y)}
\bigl(g_0(y)\wh v_0(y),g_1(x)\wh v_1(x)\bigr)_\calK dx\,dy,
\end{multline*}
as required.
\end{proof}
This lemma shows that $\DOI(a)$ can be alternatively defined through the integral kernel of $a$. 
Since the integral kernel is independent of the choice of the Schmidt series representation \eqref{c3b}, 
our definition of $\DOI(a)$ is also independent of this choice. 

\begin{lemma}\label{lma.c2}
For any finite rank operator $a$, one has (with $A$ as in \eqref{c3a}) 
\[
\norm{\DOI(a)}_\calB\leq A\norm{a}_\calB.
\label{c7}
\]
\end{lemma}
\begin{proof}
Let $a$ be as in \eqref{c3b}; 
observe that $\max_n s_n=\norm{a}_\calB$.
The sesquilinear form of $\DOI(a)$ is 
$$
(\DOI(a)v_0,v_1)
=
\sum_{n=1}^N s_n (G_0\varphi_n(H_0)^*v_0,G_1\psi_n(H_1)^* v_1).
$$
Applying Cauchy-Schwarz and Theorem~\ref{thm.b2}, we can estimate this form as follows:
\begin{multline*}
\abs{(\DOI(a)v_0,v_1)}
\leq 
\sum_{n=1}^N s_n \norm{G_0\varphi_n(H_0)^*v_0}_\calK\norm{G_1\psi_n(H_1)^*v_1}_\calK
\\
\leq
\norm{a}_\calB
\biggl(\sum_{n=1}^N \norm{G_0\varphi_n(H_0)^*v_0}_\calK^2\biggr)^{1/2}
\!\biggl(\sum_{n=1}^N \norm{G_1\psi_n(H_1)^*v_1}_\calK^2\biggr)^{1/2}
\!\!\leq 
A \norm{a}_\calB \norm{v_0}_\calH\norm{v_1}_\calH,
\end{multline*}
as required. 
\end{proof}

\begin{lemma}\label{lma.c3}
Let $a_n$, $a$ be finite rank operators such that $a_n\to a$ $*$-weakly.
Then $\DOI(a_n)\to\DOI(a)$ $*$-weakly.
\end{lemma}
\begin{proof}
By linearity it suffices to consider the case $a_n\to0$ $*$-weakly. 
By Lemma~\ref{lma.c1}, we have
\[
(\DOI(a_n)v_0,v_1)
=
\int_\bbR\int_\bbR
a_n(x,y)(g_0(y)\wh v_0(y),g_1(x)\wh v_1(x))_\calK dx\,dy,
\label{c8}
\]
where
$$
\int_\bbR \norm{g_j(x)\wh v_j(x)}_\calK^2 dx
\leq
\norm{G_j}^2_{\Smooth(H_j)}\norm{v_j}_\calH^2, \quad j=0,1.
$$
Let $\{e_\ell\}$ be an orthonormal basis in $\calK$. Denote
$$
F_{j,\ell}(x)=(g_j(x)\wh v_j(x),e_\ell),\quad x\in\bbR,\quad j=0,1,
$$
and consider the operator $K$ in $L^2(\bbR)$ with the integral kernel
$$
K(x,y)=\sum_\ell F_{0,\ell}(x)\overline{F_{1,\ell}(y)}.
$$
This operator is trace class, because
$$
\sum_{\ell}\norm{F_{0,\ell}}_{L^2}\norm{F_{1,\ell}}_{L^2}
\leq 
\bigl(\sum_{\ell}\norm{F_{0,\ell}}_{L^2}^2\bigr)^{1/2}
\bigl(\sum_{\ell}\norm{F_{1,\ell}}_{L^2}^2\bigr)^{1/2}
\leq 
\norm{g_0\wh v_0}_{L^2}
\norm{g_1\wh v_1}_{L^2}
<\infty.
$$
Now let us expand the inner product in \eqref{c8} as
$$
(g_0(y)\wh v_0(y),g_1(x)\wh v_1(x))_\calK
=
\sum_\ell
(g_0(y)\wh v_0(y),e_\ell)_\calK
(e_\ell,g_1(x)\wh v_1(x))_\calK;
$$
this yields
$$
(\DOI(a_n)v_0,v_1)
=
\sum_\ell \int_\bbR\int_\bbR a_n(x,y) F_{0,\ell}(y)\overline{F_{1,\ell}(x)}dx\,dy
=
\Tr (a_n K). 
$$
By our assumption on $*$-weak convergence, we have
$\Tr(a_n K)\to0$ as $n\to\infty$, and therefore $\DOI(a_n)\to0$ $*$-weakly. 
\end{proof}

\subsection{$\DOI(a)$ for bounded and compact $a$}

In the previous subsection, we have defined the map 
\[
\DOI: \calB(L^2(\bbR))\to \calB(\calH)
\label{c9again}
\]
on the set of all finite rank operators; we have checked this map is bounded in the operator norm
and continuous with respect to the $*$-weak convergence. 
Since finite rank operators are $*$-weakly dense in the set of bounded operators, 
we can extend this map (by $*$-weak continuity) onto the whole set $\calB(L^2(\bbR))$.

\begin{lemma}\label{lma.c4}
The map \eqref{c9again}, extended as explained above, is bounded with 
respect to the operator norm, and the operator norm bound \eqref{c2} holds true. 
The property \eqref{c1a} also holds for any bounded $a$.  
\end{lemma}
\begin{proof}
Let $P_n$ be a sequence of finite rank orthogonal projections in $L^2(\bbR)$
such that $P_n\to I$ strongly as $n\to\infty$.
Denote $a_n=P_n aP_n$.
Then $a_n\to a$ $*$-weakly and $\norm{a_n}_\calB\leq\norm{a}_\calB$ for all $n$. 
Using the bound \eqref{c7} for finite rank operators, we obtain
$$
\norm{\DOI(a)}_\calB
\leq 
\liminf_{n\to\infty}\norm{\DOI(a_n)}_\calB
\leq 
A \liminf_{n\to\infty}\norm{a_n}_\calB
\leq 
A\norm{a}_\calB.
$$
Finally, it is clear that the property \eqref{c1a} is preserved under the weak limits. 
\end{proof}

Recall that the class $\Smooth_\infty(H)\subset\Smooth(H)$ is defined by the 
additional compactness assumption \eqref{b10}.
\begin{lemma}\label{lma.c5}
Assume that $a\in\Sch_\infty$, 
$G_0\in\Smooth(H_0)$ and $G_1\in \Smooth(H_1)$, and 
suppose in addition that either $G_0\in\Smooth_\infty(H_0)$ or $G_1\in\Smooth_\infty(H_1)$ (or both). 
Then $\DOI(a)\in\Sch_\infty$.
\end{lemma}
\begin{proof}
Consider the case $G_0\in\Smooth_\infty(H_0)$. 
By Lemma~\ref{lma.c4}, it suffices to check that $\DOI(a)\in\Sch_\infty$
for all finite rank $a$. By linearity, it suffices to consider rank one operators $a$.
Let $a(x,y)=\psi(x)\overline{\varphi(y)}$; then 
$$
\DOI(a)=(G_1\psi(H_1)^*)^* G_0\varphi(H_0)^*.
$$
Here $G_1\psi(H_1)^*$ is bounded by Theorem~\ref{thm.b1} 
and $G_0\varphi(H_0)^*$ is compact by Lemma~\ref{lma.b3}. 
This gives the compactness of $\DOI(a)$. 
The case $G_1\in\Smooth_\infty(H_1)$ is considered in the same way. 
\end{proof}

\subsection{$\DOI(a)$ for $a\in\Sch_p$}

\begin{theorem}\label{thm.c6}
Let $p$, $q$, $r$ be finite positive numbers such that $\frac1p=\frac1q+\frac1r$. 
Let $G_0\in\Smooth_{q}(H_0)$, $G_1\in\Smooth_{r}(H_1)$.
Then for all $a\in\Sch_p$, we have $\DOI(a)\in\Sch_p$ and the Schatten norm bound
\eqref{c3} holds true.

This extends to $q=\infty$ (resp. $r=\infty$) if one replaces $\Smooth_q(H_0)$ (resp. $\Smooth_r(H_1)$)
by $\Smooth(H_0)$ (resp. $\Smooth(H_1)$). 
\end{theorem}
\begin{proof}
First let us consider the case of finite $q$, $r$. 
By a density argument it suffices to prove \eqref{c3} for finite rank operators $a$.  Let $a$ be given by its Schmidt series \eqref{c3b}, so 
$$
\norm{a}_p^p=\sum_{n=1}^N s_n^p
\quad\text{ and }\quad
\DOI(a)=\sum_{n=1}^N s_n (G_1\psi_n(H_1)^*)^*G_0\varphi_n(H_0)^*.
$$
We write $\DOI(a)$ in a factorised form:
$$
\DOI(a)=T_1^*T_0,
$$
where the maps $T_j: \calH\to\ell^2(\bbN;\calK)$, $j=0,1$ are defined by 
\begin{align*}
(T_0 u)_n =s_n^{p/q} G_0\varphi_n(H_0)^* u, \quad n\in\bbN,
\\
(T_1 u)_n =s_n^{p/r} G_1\psi_n(H_1)^* u, \quad n\in\bbN.
\end{align*}
Our aim is to show that $T_0\in \Sch_q$ and $T_1\in\Sch_r$ with the norm bounds
\begin{align}
\norm{T_0}_{q}&\leq \norm{a}_p^{p/q}\norm{G_0}_{\Smooth_q(H_0)}, 
\label{c8b}
\\
\norm{T_1}_{r}&\leq \norm{a}_p^{p/r}\norm{G_1}_{\Smooth_r(H_1)}.
\label{c8c}
\end{align}
From \eqref{c8b} and \eqref{c8c} the required result follows immediately by an application
of the ``H\"older inequality for $\Sch_p$ classes": 
$$
\norm{\DOI(a)}_p=\norm{T_1^*T_0}_p\leq \norm{T_1}_q\norm{T_0}_r\leq\norm{a}_p A_{q,r}.
$$
Let us prove the bound \eqref{c8b}; the second bound \eqref{c8c} is considered in the same way.

\emph{Case 1: $0<q\leq 2$.}
Consider the operator
$$
T_0^*T_0=\sum_{n=1}^N s_n^{2p/q}(G_0\varphi_n(H_0)^*)^*G_0\varphi_n(H_0)^*. 
$$
We use the ``triangle inequality'' for $\norm{\cdot}_{q/2}^{q/2}$, see \eqref{b7aa}:
$$
\norm{T_0}_q^q
=
\norm{T_0^*T_0}_{q/2}^{q/2}
\leq
\sum_{n=1}^N s_n^p\norm{(G_0\varphi_n(H_0)^*)^*G_0\varphi_n(H_0)^*}_{q/2}^{q/2}
=
\sum_{n=1}^N s_n^p\norm{G_0\varphi_n(H_0)^*}_q^q.
$$
By the definition of the $\Sch_{q}$-valued smoothness, 
$$
\norm{G_0\varphi_n(H_0)^*}_{q}
\leq
\norm{G_0}_{\Smooth_{q}(H_0)}\norm{\varphi_n}_{L^2}
=
\norm{G_0}_{\Smooth_{q}(H_0)},
$$
since $\varphi_n$ are normalised in $L^2$. 
Putting this together, we obtain the bound \eqref{c8b}.

\emph{Case 2: $q\geq2$.}
Here we use complex interpolation between the cases $q=2$ and $q=\infty$ and employ Lemma~\ref{lma.b6}.

Let $G_0(z)$ be the analytic family as in Lemma~\ref{lma.b6} with $G=G_0$ and $H=H_0$. 
For $0\leq \Re z\leq 1$, let $T_0(z): \calH\to\ell^2(\bbN;\calK)$
be defined by 
$$
(T_0(z)u)_n=s_n^{pz/2}G_0(z)\varphi_n(H_0)^*u,\quad n\geq 1.
$$
Let us compute the operator norm of $T_0(z)$. Using Theorem~\ref{thm.b2}, we obtain
\begin{multline*}
\sum_{n=1}^N \norm{(T_0(z)u)_n}_{\calK}^2
\leq
\norm{a}_\calB^{p\Re z}
\sum_{n=1}^N \norm{G_0(z)\varphi_n(H_0)^*u}_{\calK}^2
\\
\leq
\norm{a}_\calB^{p\Re z}
\norm{G_0(z)}_{\Smooth(H_0)}^2\norm{u}_\calH^2.
\end{multline*}
Thus, $T_0(z)$ is bounded in the operator norm for all $0\leq \Re z\leq 1$
and 
$$
\norm{T_0(z)}_\calB\leq \norm{G_0(z)}_{\Smooth(H_0)} \leq 1, \quad \Re z=0.
$$
Next, for $\Re z=1$ the operator $T_0(z)$ is Hilbert-Schmidt. 
Indeed, using the estimates of Lemma~\ref{lma.b6}, we obtain
\begin{multline*}
\norm{T_0(z)}^2_2
=
\sum_{n=1}^N s_n^p \norm{G_0(z)\varphi_n(H_0)^*}_2^2
\\
\leq
\sum_{n=1}^N s_n^p \norm{G_0(z)}_{\Smooth_2(H_0)}^2
\leq
\norm{a}_p^p \norm{G_0}_{\Smooth_{q}(H_0)}^{q}, 
\quad \Re z=1.
\end{multline*}
Further, it is straighforward to see that $T_0(z)$ is analytic in $0<\Re z<1$, 
operator norm continuous for $0\leq \Re z\leq 1$ and $T_0(2/q)=T_0$. 
By Hadamard's three lines theorem for Schatten classes \cite[Thm. III.13.1]{Gohberg-Krein}, we obtain
$$
\norm{T_0}_{q}
=
\norm{T_0(2/q)}_{q}
\leq
\biggl(
\norm{a}_p^{p/2}\norm{G_0}^{q/2}_{\Smooth_{q}(H_0)}
\biggr)^{2/q}
=
\norm{a}_p^{p/q}\norm{G_0}_{\Smooth_{q}(H_0)},
$$
as required. 

Finally, let us briefly discuss the case $r=\infty$, $q=p$ (the case $q=\infty$, $r=p$ is considered in the same way). 
Here we set 
\begin{align*}
(T_0 u)_n &=s_nG_0\varphi_n(H_0)^* u, \quad n\in\bbN,
\\
(T_1 u)_n &= G_1\psi_n(H_1)^* u, \quad n\in\bbN.
\end{align*}
Now we have an operator norm bound for $T_1$ by Theorem~\ref{thm.b2} and the $\Sch_q$-norm bound
for $T_0$ by the same argument as above (considering separately the $q\leq 2$ and $q\geq2$ cases). 
Combining these bounds, we obtain
$$
\norm{\DOI(a)}_p=\norm{T_1^*T_0}_p\leq \norm{T_1}_\calB\norm{T_0}_p\leq\norm{a}_p A_{p,\infty},
$$
as required.
\end{proof}

\section{The map $f\mapsto\widecheck{f}$ }\label{sec.d}

\subsection{Overview}
As in the Introduction, for a function $f:\bbR\to\bbC$, we denote by $\wf$ the 
divided difference
$$
\wf(x,y):=\frac{f(x)-f(y)}{x-y}, \quad x,y\in\bbR.
$$
By a slight abuse of notation, we also denote by 
$\wf$ the operator in $L^2(\bbR)$ with the integral kernel $\wf(x,y)$. 
Of course, this definition requires some assumptions on $f$; we will be 
more precise below.  
Our aim in this section is to establish the boundedness of 
$f\mapsto\wf$ as a map from $\BMO(\bbR)$ to $\calB(L^2(\bbR))$ and from $B_{p,p}^{1/p}(\bbR)$ to $\Sch_p$. 
The content of this section is probably well-known to specialists; we just need to recall the required results in
notation convenient for the next section.

\subsection{Preliminaries on $\BMO$}

The Hardy space $H^p(\bbC_+)$, $p\geq1$, is defined in the standard way as the space of all 
analytic functions $u$ in the upper half-plane such that the norm
$$
\norm{u}_{H^p(\bbC_+)}^p
=
\sup_{y>0}\int_{-\infty}^\infty \abs{u(x+iy)}^pdx
$$
is finite. As usual, we identify the function $u\in H^p(\bbC_+)$ with its
boundary values $u(x)=u(x+i0)$, which exist for a.e. $x\in\bbR$.
The spaces $H^p(\bbC_-)$ are defined analogously. 
In fact, we will only need the cases $p=1$ and $p=2$.

The space $\BMO(\bbR)$ (bounded mean oscillation) consists of all locally integrable functions $f$ on $\bbR$
such that the following supremum over all bounded intervals $I\subset \bbR$
is finite:
$$
\sup_{I}\jap{\abs{f-\jap{f}_I}}_I<\infty, 
\quad
\jap{f}_I=\abs{I}^{-1}\int_I f(x)dx. 
$$
Observe that this supremum vanishes on constant functions. 
Strictly speaking, the elements of 
$\BMO(\bbR)$ should be regarded not as functions but as
equivalence classes $\{f+\const\}$; in practice, we will deal 
with individual functions but bear in mind that an arbitrary constant 
can be added to a function without affecting the BMO norm.
Observe that for constant functions, the kernel \eqref{d1} 
vanishes identically.

Functions in $\BMO(\bbR)$ belong to $L^p(-R,R)$ for any $R>0$ and any $p<\infty$, 
but not for $p=\infty$: they may have logarithmic singularities. 
These functions also satisfy \cite[Theorem VI.1.2]{Garnett}
\[
f\in\BMO(\bbR)
\quad\Rightarrow\quad
\int_{-\infty}^\infty \frac{\abs{f(x)}}{1+x^2}dx<\infty.
\label{d2}
\]
Fefferman's duality theorem \cite[Theorem VI.4.4]{Garnett}
says that for any $f\in\BMO(\bbR)$, the 
linear functional on $H^1(\bbC_+)$, 
\[
T_f(u):=\int_{-\infty}^\infty f(x)u(x)dx,
\label{d3}
\]
defined initially on a suitable dense set of functions $u$, extends to 
the whole space $H^1(\bbC_+)$ as a bounded linear functional
and that conversely, any bounded linear functional on $H^1(\bbC_+)$
can be realised in this way with some $f\in\BMO(\bbR)$.
The norm of $T_f$ in the dual space $H^1(\bbC_+)^*$ will be denoted by
$\norm{T_f}_{H^1(\bbC_+)^*}$.

A minor technical issue here is that the integral in \eqref{d3} need not
make sense for all $f\in\BMO$ and all $u\in H^1(\bbC_+)$. 
This explains the need for using certain dense sets of $f$'s and $u$'s in what follows. 

There are many equivalent ways to define a norm on $\BMO(\bbR)$. 
We choose the one directly related to Fefferman's duality theorem. 
For $f\in\BMO(\bbR)$, we set
$$
\norm{f}_{\BMO}
:=
\max\{\norm{T_f}_{H^1(\bbC_+)^*}, \norm{T_{\overline{f}}}_{H^1(\bbC_+)^*}\}.
$$
We will say that $f_n\to f$  $*$-weakly in $\BMO(\bbR)$, if we have the weak 
convergence of linear functionals $T_{f_n}\to T_f$ and 
$T_{\overline{f}_n}\to T_{\overline{f}}$
on $H^1(\bbC_+)$. 

We will denote by $\calR$ the set of all bounded rational functions of $x\in\bbR$: 
$$
\calR=\{p/q: p, q \text{ polynomials}, \deg p\leq \deg q,\, q(x)\not=0 \text{ for } x\in\bbR\}.
$$
The subspace $\CMO(\bbR)\subset\BMO(\bbR)$ (continuous mean oscillation) 
is the closure of all rational functions $\calR$ in $\BMO(\bbR)$. 
(Alternatively, one can define $\CMO$ as the closure in $\BMO$ of the set of all functions of the form $f+\const$, 
$f\in C_0^\infty(\bbR)$.)
\begin{remark*}
The space $\CMO(\bbR)$ is slightly  smaller than the more commonly used space $\VMO(\bbR)$ (vanishing mean oscillation) of functions
defined by the condition
$$
\lim_{\epsilon\to 0} \sup_{|I|\leq\epsilon}\jap{\abs{f-\jap{f}_I}}_I = 0
$$
(see, e.g., \cite[Section 2A]{Rochberg}).
Roughly speaking, the functions in $\VMO$ must be ``more regular than BMO'' locally, 
while the functions in $\CMO$ must be ``more regular than BMO''  
both locally and at infinity. 
For example, the function $\log(1+x^2)$ belongs to $\VMO$ but not to $\CMO$. 
\end{remark*}

Finally, we will need the following

\begin{lemma}\label{lma.d1}
The set $\calR$ of rational functions is dense in $\BMO(\bbR)$ with respect to $*$-weak convergence. 
\end{lemma}

The proof is given in the Appendix. 

\subsection{Besov spaces}

Let $w\in C_0^\infty(\bbR)$, $w\geq0$, be a function with $\supp w\subset [1/2,2]$ and such that
$$
\sum_{j\in\bbZ}w_j(x)=1, \quad x>0, \quad \text{ where }w_j(x)=w(x/2^j). 
$$
The (homogeneous) Besov class $B_{p,p}^{1/p}(\bbR)$ is defined as the space of tempered distributions $f$ on $\bbR$ such that
\[
\norm{f}^p_{B_{p,p}^{1/p}}:=
\sum_{j\in\bbZ} 2^j 
\bigl(\norm{f*\wh w_j}_{L^p(\bbR)}^p+\norm{f*\overline{\wh w_j}}_{L^p(\bbR)}^p\bigr)
<\infty.
\label{A1}
\]
Here $*$ is the convolution and $\wh w_j$ is the Fourier transform of $w_j$, 
$$
\wh w_j(t)=\frac1{2\pi} \int_\bbR e^{-itx}w_j(x)dx. 
$$
Observe that according to this definition, any polynomial $f$ belongs to $B_{p,p}^{1/p}(\bbR)$
(as the Fourier transform of a polynomial is supported at the origin). 
In the context of this paper, we consider $B_{p,p}^{1/p}(\bbR)\cap \BMO(\bbR)$, which 
reduces an arbitrary polynomial to an arbitrary constant.  

The definition of $B_{p,p}^{1/p}$ is independent of the choice of the function $w$. 
However, the precise value of $\norm{f}_{B^{1/p}_{p,p}}$ will, of course, depend
on this choice.

\subsection{Discussion: $\wf$ and Hankel operators}

Recall that the orthogonal projection $P_+:L^2(\bbR)\to H^2(\bbC_+)$ 
onto the Hardy class is given by 
$$
(P_+ u)(x)=-\frac1{2\pi i}\lim_{\eps\to+0}\int_{-\infty}^\infty \frac{u(y)}{x-y+ i\eps}dy,
\quad
u\in L^2(\bbR).
$$
Comparing this with \eqref{d1}, we see that, at least for smooth bounded functions $f$, 
the operator $\wf$ can be identified with the commutator $2\pi i[P_+,M_f]$, 
where $M_f$ is the operator of multiplication by $f$. 
Further, formally we have (denoting $P_-=I-P_+$)
$$
\frac1{2\pi i}\wf
=
P_+M_f-M_fP_+
=
P_+M_fP_--P_-M_fP_+.
$$
In accordance with this, we define $\wf$ initially via the sesquilinear form
(denoting $u_\pm=P_\pm u$)
\[
(\wf u,v)=2\pi i(fu_-,v_+)-2\pi i (fu_+,v_-), 
\quad
u,v\in L^\infty_\comp(\bbR).
\label{d4a}
\]
Let us explain why the inner products in \eqref{d4a} are well-defined. 
Since $u\in L^\infty_\comp(\bbR)$, for some $R>0$ we have
$$
\int_{-R}^R (\abs{u_+(x)}^p+\abs{u_-(x)}^p)dx<\infty, \quad \forall p<\infty,
$$
and
$$
\abs{u_+(x)}+\abs{u_-(x)}\leq C\abs{x}^{-1}, \quad \abs{x}>R,
$$
and similar bounds hold for $v_\pm$.
Recall also that 
$f\in L^p(-R,R)$ for any $R>0$ and any $p<\infty$, and the 
integral \eqref{d2} converges. Putting this together, we see that the integrals
$$
\int_{-\infty}^\infty f(x)u_-(x)\overline{v_+(x)}dx
\quad \text{ and }\quad
\int_{-\infty}^\infty f(x)u_+(x)\overline{v_-(x)}dx
$$
converge absolutely, and so the inner products in \eqref{d4a} are well defined. 
Although these inner products need not make sense for arbitrary $u,v\in L^2$, below we will 
see that $(\wf u,v)$ is bounded in $u,v$ in the $L^2$ norm, and therefore $\wf$ 
extends as a bounded operator to $L^2$. 

Further, we have 
\[
\frac1{2\pi i}\wf
=
P_+fP_--P_-fP_+
=
\begin{pmatrix}
0 & P_+fP_-
\\
-P_-fP_+ & 0
\end{pmatrix}
\label{d4c}
\]
with respect to the orthogonal decomposition $L^2(\bbR)=\Ran P_+\oplus\Ran P_-$.
This gives an immediate (and well-known) connection with Hankel operators. 
For $f\in\BMO(\bbR)$, the Hankel operator  $H(f)$ is defined by  
$$
H(f): H^2(\bbC_+)\to H^2(\bbC_-), 
\quad
H(f) u=P_-(fu), \quad u\in H^2(\bbC_+).
$$
Thus, $P_-fP_+$ is exactly the Hankel operator $H(f)$ but defined on the wider space $L^2(\bbR)$; 
in particular, the operator norm (and all Schatten norms) of the operators $P_-fP_+$ and $H(f)$
coincide. This shows that the required results on the boundedness and Schatten class properties of $\wf$ 
follow directly from the corresponding known results on Hankel operators. 
Below we make this explicit. 

\subsection{Boundedness of $\wf$}

\begin{lemma}\label{lma.d2}
\begin{enumerate}[\rm (i)]
\item
Let $f\in\BMO(\bbR)$. 
Then the sesquilinear form \eqref{d4a} satisfies 
the bound
$$
\abs{(\wf u,v)}\leq 2\pi \norm{f}_{\BMO}\norm{u}_{L^2}\norm{v}_{L^2},
\quad
u,v\in L^\infty_\comp(\bbR).
$$
Thus, $\wf$ extends to a bounded operator on $L^2(\bbR)$. 
Further, one has 
\[
\norm{\wf}=2\pi \norm{f}_{\BMO}. 
\label{d5}
\]
\item
If $f_n\to f$ $*$-weakly in $\BMO$, then $\wf_n\to\wf$ $*$-weakly in $\calB(L^2(\bbR))$. 
\end{enumerate}
\end{lemma}

\begin{proof} 
Let us first consider the quadratic form 
$$
(fu_+,v_-)
=
\int_{-\infty}^\infty f(x) u_+(x)\overline{v_-(x)}dx
$$
for $u,v\in L^\infty_\comp(\bbR)$. As already discussed, the integral here converges absolutely. 
Further, since $u_+,\overline{v_-}\in H^2(\bbC_+)$,
we have $u_+\overline{v_-}\in H^1(\bbC_+)$ and so 
$$
(fu_+,v_-)=T_f(u_+\overline{v_-}).
$$
It follows that 
$$
\abs{(fu_+,v_-)}=\abs{T_f(u_+\overline{v_-})}
\leq
\norm{T_f}_{H^1(\bbC_+)^*}\norm{u_+\overline{v_-}}_{H^1(\bbC_+)}
\leq
\norm{T_f}_{H^1(\bbC_+)^*}\norm{u}_{L^2}\norm{v}_{L^2},
$$
which can be written as
$$
\norm{P_-fP_+}\leq \norm{T_f}_{H^1(\bbC_+)^*}.
$$
Further, since (see e.g. \cite[Exercise II.1]{Garnett}) any function in $H^1(\bbC_+)$ can be represented as
$u_+\overline{v_-}$  with 
$$
\norm{u_+\overline{v_-}}_{H^1(\bbC_+)}=\norm{u_+}_{L^2}\norm{v_-}_{L^2},
$$ 
it is easy to see that in fact we have the equality of the norms,
$$
\norm{P_-fP_+}=\norm{T_f}_{H^1(\bbC_+)^*}.
$$
Similarly, 
$$
(fu_-,v_+)=\overline{T_{\overline{f}}(v_+\overline{u_-})}
\text{ and }
\norm{P_+fP_-}= \norm{T_{\overline{f}}}_{H^1(\bbC_+)^*}.
$$
Now by  \eqref{d4c} we obtain
$$
\frac1{2\pi}\norm{\wf}
=
\max \{\norm{P_-fP_+},\norm{P_+fP_-}\}
=
\max\{\norm{T_f}_{H^1(\bbC_+)^*}, \norm{T_{\overline{f}}}_{H^1(\bbC_+)^*}\}
=
\norm{f}_{\BMO},
$$
according to our definition of the BMO norm.

This argument also shows that if $f_n\to0$ $*$-weakly
in BMO, then 
$$
(f_nu_+,v_-)=T_{f_n}(u_+\overline{v_-})\to0
$$
and
$$
(f_nu_-,v_+)=\overline{T_{\overline{f_n}}(v_+\overline{u_-})}\to0,
$$
which yields (ii).
\end{proof}

\begin{lemma}\label{lma.d3}
If $f\in\calR$, then the operator $\wf$ has a finite rank.
If $f\in\CMO(\bbR)$, then the operator $\wf$ is compact. 
\end{lemma}
\begin{proof}
Let $f(x)=(x-z_0)^{-1}$, $\Im z_0\not=0$. 
Then 
\[
\wf(x,y)=\frac{(x-z_0)^{-1}-(y-z_0)^{-1}}{x-y}=-(x-z_0)^{-1}(y-z_0)^{-1}
\label{d7}
\]
so $\wf$ is a rank one operator. Differentiating \eqref{d7} 
$m$ times with respect to $z_0$, one checks that $\wf$ is finite rank for $f(x)=(x-z_0)^{-m-1}$. 
By partial fraction decomposition, we get that $\wf$ is finite rank 
for any rational $f$. 

Now let $f\in\CMO(\bbR)$; approximating $f$ by rational functions in $\BMO$ norm,
we obtain in view of \eqref{d5} an approximation of $\wf$ by finite rank operators in the operator norm.
Thus, $\wf$ is compact. 
\end{proof}

\subsection{Schatten class properties of $\wf$}
Below we state Peller's characterisation of  Hankel operators of Schatten class in a form convenient for us.
For the proofs and the history, see \cite[Chapter 6]{Peller}.
\begin{proposition}\cite[Theorem 6.7.4]{Peller}\label{prp.d4}
For any $0<p<\infty$, there exist constants $c_1(p)<C_1(p)$ such that 
for all $f\in B_{p,p}^{1/p}(\bbR)\cap\BMO(\bbR)$, 
\[
c_1(p)^p\norm{f}_{B_{p,p}^{1/p}}^p
\leq 
\norm{H(f)}_{p}^p+\norm{H(\overline{f})}_{p}^p
\leq 
C_1(p)^p\norm{f}_{B_{p,p}^{1/p}}^p.
\label{d8}
\]
\end{proposition}

\begin{remark*}
Of course, the constants $C_1(p)$ and $c_1(p)$ depend on the choice
of the functional $\norm{\cdot}_{B_{p,p}^{1/p}}$ in $B_{p,p}^{1/p}(\bbR)$. 
The bounds \eqref{d8} are not explicitly stated in \cite[Theorem 6.7.4]{Peller}, but
are obtained in the  proof of that theorem. 
\end{remark*}

\begin{lemma}\label{lma.d4}
For any $0<p<\infty$, one has
$$
(2\pi)^{-p}\norm{\wf}_p^p=\norm{H(f)}_{p}^p+\norm{H(\overline{f})}_{p}^p,
\quad 
f\in B_{p,p}^{1/p}(\bbR).
$$
Thus, we have the estimates
$$
c_1(p)\norm{f}_{B_{p,p}^{1/p}}
\leq 
(2\pi)^{-1}
\norm{\wf}_p
\leq 
C_1(p)\norm{f}_{B_{p,p}^{1/p}},
\quad 
f\in B_{p,p}^{1/p}(\bbR),
$$
with the constants as in Proposition~\ref{prp.d4}.
\end{lemma}
\begin{proof}
By \eqref{d4c}, we have 
$$
\frac1{2\pi i}\wf
=
\begin{pmatrix}
0 & (P_- \overline{f} P_+)^*
\\
-P_- fP_+ & 0
\end{pmatrix}
\text{ in $L^2(\bbR)=\Ran P_+\oplus \Ran P_-$.}
$$
Now the required result follows from the fact that (by a simple calculation) 
$$ 
\norm{X}_p^p=\norm{A}_{p}^p+\norm{B}_p^p
\quad \text{ for }
X=
\begin{pmatrix}
0&A\\
B&0
\end{pmatrix}.
\qedhere
$$
\end{proof}

\section{The map $f\mapsto D(f)$ }\label{sec.e}

\subsection{Overview}
In this section we put together all the components prepared so far. 
Throughout this section, $H_0$ and $H_1$ are self-adjoint operators in a Hilbert space $\calH$
and $G_0$, $G_1$ are linear operators from $\calH$ to $\calK$ such that 
\[
G_0\in\Smooth(H_0) \quad\text{ and }\quad G_1\in\Smooth(H_1).
\label{ee0a}
\]
We assume that
\[
H_1-H_0=G_1^*G_0
\label{ee0}
\]
in the sense to be made precise later. 
We consider the map $f\mapsto D(f)$ in an abstract fashion,
as a linear map from some function spaces to some spaces of operators.
Our aim is to prove Theorems~\ref{thm.a1} and \ref{thm.a2}, which are restated
more precisely as Theorems~\ref{thm.e1} and \ref{thm.e3} below. 
The key step is the use of the Birman-Solomyak formula \eqref{BS},
which allows us to use the results of Sections~\ref{sec.c} and \ref{sec.d}.

\subsection{Preliminaries}

First we should explain that the identity \eqref{ee0} will be understood in the sesquilinear form sense:
\[
(u,H_1v)-(H_0u,v)=(G_0u,G_1v), \quad u\in\Dom(H_0), \quad v\in\Dom(H_1).
\label{e0}
\]
Next, since functions $f\in\BMO(\bbR)$ need not be bounded, the operators $f(H_0)$ and $f(H_1)$
are in general unbounded for such $f$. Thus, the definition of $D(f)$ requires some care. 
Similarly to \eqref{e0}, we define the sesquilinear form of $D(f)$ as follows: 
\[
d_f[u,v]:=(u,\overline{f}(H_1)v)-(f(H_0)u,v), \quad u\in\Dom(f(H_0)), \quad v\in\Dom(f(H_1)). 
\label{e0a}
\]
Obviously, for bounded functions $f$ one can define $D(f)$ directly as a bounded operator on $\calH$ 
and in this case we have
\[
d_f[u,v]=(D(f)u,v), \quad u\in \Dom f(H_0),\quad v\in\Dom f(H_1). 
\label{e0b}
\]
In what follows we will prove that for any $f\in\BMO(\bbR)$ the sesquilinear form $d_f[u,v]$ 
is bounded and therefore \eqref{e0b} holds with some 
bounded operator  $D(f)$ in $\calH$.

We denote
$$
R_0(z)=(H_0-z)^{-1}, \quad R_1(z)=(H_1-z)^{-1}, \quad \Im z\not=0.
$$
We will need the resolvent identity for operators satisfying \eqref{e0}; it can be written in two 
alternative forms:
\begin{align}
R_1(z)-R_0(z)&=-(G_1R_1(\overline{z}))^*G_0R_0(z),
\label{e0c}
\\
R_1(z)-R_0(z)&=-(G_0R_0(\overline{z}))^*G_1R_1(z),
\label{e0d}
\end{align}
for any $\Im z\not=0$. 
  
First we give a simple statement reducing the analysis of $D(f)$ to the absolutely continuous
subspaces of $H_0$ and $H_1$.   
  
\begin{proposition}\label{prp.e1}
Assume \eqref{ee0a} and  \eqref{e0}. 
Then for the quadratic form $d_f$, defined by \eqref{e0b}, we have 
$d_f[u,v]=0$ if $u\in \calH^{(\sing)}(H_0)\cap \Dom (f(H_0))$ or
$v\in \calH^{(\sing)}(H_1)\cap \Dom (f(H_1))$ (or both). 
\end{proposition}
\begin{proof}
Suppose $u\in\calH^{(\sing)}(H_0)$; then for any $\Im z\not=0$ we have
$GR_0(z)u=0$ and therefore, by the resolvent identity \eqref{e0c}, 
$$
R_1(z)u=R_0(z)u. 
$$
By Stone's formula \cite[Theorem VII.13]{ReSi}, this implies that the corresponding 
two spectral measures coincide on $u$: 
$$
E_{H_0}(\Delta)u=E_{H_1}(\Delta)u, \quad \forall \Delta\subset \bbR. 
$$
It follows that 
$$
(f(H_0)u,v)=(u,\overline{f}(H_1)v)
$$
whenever both sides are well-defined, i.e. whenever $u\in \Dom f(H_0)$ and $v\in \Dom f(H_1)$. 
This is the equality $d_f[u,v]=0$ written in a different form. 

The case $v\in \calH^{(\sing)}(H_1)$ is considered in the same way, by using 
the resolvent identity in the form \eqref{e0d}. 
\end{proof}

The above proposition is well known in scattering theory as the statement that under the assumptions
\eqref{ee0a}, \eqref{ee0}, the singular parts of $H_0$ and $H_1$ coincide. 
As a consequence of this proposition,  when dealing with the sesquilinear from $d_f[u,v]$, 
it suffices to consider $u\in\calH^{\ac}(H_0)$ and $v\in\calH^{\ac}(H_1)$. 
In fact, the argument of Proposition~\ref{prp.e1} also shows that these absolutely 
continuous subspaces coincide: $\calH^{\ac}(H_0)=\calH^{\ac}(H_1)$, although we will not need this.

\subsection{The norm bound for $D(f)$}
\begin{theorem}\label{thm.e1}
For any $f\in\BMO(\bbR)$ and for dense sets of $u\in\calH^{\ac}(H_0)$, $v\in\calH^{\ac}(H_1)$, 
the form \eqref{e0a} satisfies the bound
\[
\abs{d_f[u,v]}\leq 2\pi A\norm{f}_{\BMO}\norm{u}_\calH\norm{v}_\calH, 
\label{e1b}
\]
where $A$ is the constant \eqref{c3a}.
Thus, the form $d_f[u,v]$ corresponds to a bounded operator $D(f)$ on $\calH$
in the sense of \eqref{e0b}, 
and $D(f)$ satisfies the norm bound 
$$
\norm{D(f)}_\calB\leq 2\pi A\norm{f}_{\BMO(\bbR)}.
$$
If $f_n\to f$ $*$-weakly in $\BMO(\bbR)$, then $D(f_n)\to D(f)$ $*$-weakly in $\calB(\calH)$.
\end{theorem}

\begin{proof}
We will prove the bound \eqref{e1b} 
for all $u\in L^\infty_\comp(H_1)$, $v\in L^\infty_\comp(H_0)$ (see Section~\ref{sec.b1} for the definition of $L^\infty_\comp(H)$). 
Since $L^\infty_\comp(H_j)$ is dense in $\calH^{\ac}(H_j)$, $j=0,1$, this will suffice.

Since $v\in L^\infty_\comp(H_1)$,  the measure $(E_{H_1}(\cdot)u,v)$
is absolutely continuous and the function 
$$
a(\lambda):=\frac{d(E_{H_1}(\lambda)u,v)}{d\lambda}
$$
is in $L^2_\comp(\bbR)$. It follows that (here $P_\eps$ is the Poisson kernel \eqref{b6aa})
\begin{multline*}
(u,\overline{f}(H_1)v)
=
\int_{-\infty}^\infty f(x) a(x) dx
=
\lim_{\eps\to0+}\int_{-\infty}^\infty f(x)(P_\eps*a)(x)dx
\\
=
\frac1{2\pi i}\lim_{\eps\to0+}\int_{-\infty}^\infty f(x) 
\bigl((R_1(x+i\eps)-R_1(x-i\eps))u,v\bigr) dx.
\end{multline*}
Similarly, we obtain
$$
(f(H_0)u,v)
=
\frac1{2\pi i}\lim_{\eps\to0+}
\int_{-\infty}^\infty f(x) 
\bigl((R_0(x+i\eps)-R_0(x-i\eps))u,v\bigr) dx.
$$
Let us subtract the last two identities one from another and use the 
resolvent identity \eqref{e0c}. 
Denoting
$$
F_{u,v}(z)=(G_0R_0(z)u,G_1R_1(\overline{z})v),
\quad 
F_{u,v}^*(z)=\overline{F_{u,v}(\overline{z})},
\quad
\Im z\not=0,
$$
we obtain
\begin{multline*}
d_f[u,v]
=
-\frac1{2\pi i}
\lim_{\eps\to0+}
\int_{-\infty}^\infty f(x) 
(F_{u,v}(x+i\eps)-F_{u,v}(x-i\eps))dx
\\
=
-\frac1{2\pi i}
\lim_{\eps\to0+}
\int_{-\infty}^\infty f(x) 
(F_{u,v}(x+i\eps)
-\overline{F_{u,v}^*(x+i\eps)})dx.
\end{multline*}
By the definition \eqref{b1} of Kato smoothness, 
the functions $F_{u,v}$  and $F_{u,v}^*$ belong to $H^1(\bbC_+)$.
Thus, in notation  \eqref{d3} the previous identity can be written as
\[
d_f[u,v]
=
-\frac1{2\pi i}
\bigl(
T_f(F_{u,v})
-
\overline{T_{\overline{f}}(F^*_{u,v})}
\bigr).
\label{e4}
\]
We have
\begin{multline*}
\int_{-\infty}^\infty \abs{F_{u,v}(x+i\eps)}dx
\leq
\int_{-\infty}^\infty \norm{G_0R_0(x+i\eps)u}\norm{G_1R_1(x-i\eps)v}dx
\\
\leq
\frac{\alpha}2\int_{-\infty}^\infty \norm{G_0R_0(x+i\eps)u}^2 dx
+
\frac1{2\alpha}\int_{-\infty}^\infty \norm{G_1R_1(x-i\eps)v}^2dx,
\end{multline*}
where $\alpha>0$ is a parameter to be chosen later. Similiarly, 
$$
\int_{-\infty}^\infty \abs{F_{u,v}^*(x+i\eps)}dx
\leq
\frac\alpha2\int_{-\infty}^\infty \norm{G_0R_0(x-i\eps)u}^2 dx
+
\frac1{2\alpha}\int_{-\infty}^\infty \norm{G_1R_1(x+i\eps)v}^2dx.
$$
By the definition \eqref{b1} of Kato smoothness,  we get
\begin{multline*}
\int_{-\infty}^\infty(\abs{F_{u,v}(x+i\eps)}+ \abs{F_{u,v}^*(x+i\eps)})dx
\\
\leq
\frac{\alpha}2 (2\pi)^2\norm{G_0}_{\Smooth(H_0)}^2\norm{u}^2
+
\frac1{2\alpha}(2\pi)^2\norm{G_1}_{\Smooth(H_1)}^2\norm{v}^2.
\end{multline*}
Optimising over $\alpha$, we obtain
$$
\int_{-\infty}^\infty(\abs{F_{u,v}(x+i\eps)}+ \abs{F_{u,v}^*(x+i\eps)})dx
\leq
(2\pi)^2 A\norm{u}\norm{v}.
$$
Now coming back to \eqref{e4}, we have
\begin{multline*}
2\pi\abs{d_f[u,v]}
\leq 
\norm{T_f}_{H^1(\bbC_+)^*}
\norm{F_{u,v}}_{H^1(\bbC_+)}
+
\norm{T_{\overline{f}}}_{H^1(\bbC_+)^*}
\norm{F_{u,v}^*}_{H^1(\bbC_+)}
\\
\leq
\norm{f}_{\BMO}
(\norm{F_{u,v}}_{H^1(\bbC_+)}
+
\norm{F_{u,v}^*}_{H^1(\bbC_+)})
\leq
(2\pi)^2 \norm{f}_{\BMO}A\norm{u}\norm{v},
\end{multline*}
as required.

Finally, suppose $f_n\to f$ $*$-weakly in $\BMO$. 
Consider the identity \eqref{e4}. 
It has been proven above for $u\in L^\infty_\comp(H_0)$, $v\in L^\infty_\comp(H_1)$;
but since we already know that $d_f[u,v]$ is bounded, it extends by a limiting argument 
to all $u,v\in\calH$.  
By the definition of $*$-weak convergence in $\BMO(\bbR)$ we deduce from \eqref{e4} that 
$$
(D(f_n)u,v)\to (D(f)u,v), \quad u,v\in\calH,
$$
as required.
\end{proof}

\subsection{Birman--Solomyak formula}
Here we discuss the Birman-Solomyak formula \eqref{BS}. 
As in Section~\ref{sec.c}, we use the shorthand notation $\DOI(a)$, see \eqref{c1}. 
In our framework, the Birman-Solomyak formula becomes
\begin{theorem}\label{thm.e2}
For all $f\in\BMO(\bbR)$, the identity
\[
D(f)=\DOI(\wf)
\label{e5} 
\]
holds true. 
\end{theorem}
\begin{proof}
First let us check \eqref{e5} for $f(x)=(x-z_0)^{-1}$, $\Im z_0\not=0$.
By the resolvent identity \eqref{e0c}, we have
\begin{gather*}
D(f)=R_1(z_0)-R_0(z_0)=-(G_1R_1(\overline{z_0}))^*G_0R_0(z_0),
\\
\wf(x,y)
=
\frac{(x-z_0)^{-1}-(y-z_0)^{-1}}{x-y}
=
-(x-z_0)^{-1}(y-z_0)^{-1},
\end{gather*}
and so, by the definition \eqref{c5} of DOI,
$$
\DOI(\wf)=-(G_1(H_1-\overline{z_0})^{-1})^*G_0(H_0-z_0)^{-1}=D(f),
$$
as claimed. 
Next, if $f(x)=(x-z_0)^{-1-m}$, $m\geq0$, then the required identity follows by 
differentiating $m$ times with respect to $z_0$. 
By partial fraction decomposition, it follows that \eqref{e5} holds true 
for all $f\in\calR$. 

Now let us extend \eqref{e5} to all $f\in\BMO(\bbR)$ by using $*$-weak 
convergence. Rational functions are $*$-weak dense in $\BMO$ by Lemma~\ref{lma.d1}.
The left side of \eqref{e5} is continuous with respect to $*$-weak convergence by 
Theorem~\ref{thm.e1}.
The map $f\mapsto \wf$ is $*$-weak continuous by Lemma~\ref{lma.d2}(ii), 
and the map $\wf\mapsto\DOI(\wf)$ is $*$-weak continuous by Lemma~\ref{lma.c3}
(and because we have defined $\DOI$ to be the $*$-weak continuous extension 
from finite rank operators). 
Thus, \eqref{e5} holds true for all $f\in\BMO(\bbR)$.
\end{proof}

\subsection{Compactness and Schatten class properties of $D(f)$}

\begin{theorem}\label{thm.e2a}
Let $f\in \CMO(\bbR)$ and $G_0\in \Smooth(H_0)$, $G_1\in\Smooth(H_1)$. 
Assume in addition that either $G_0\in\Smooth_\infty(H_0)$ or 
$G_1\in \Smooth_\infty(H_1)$. Then $D(f)$ is compact. 
\end{theorem}
\begin{proof}
By Theorem~\ref{thm.e2}, it suffices to check that $\DOI(\wf)$ is compact. 
Here $\wf$ is compact by Lemma~\ref{lma.d3}.
Now the result follows from Lemma~\ref{lma.c5}.  
\end{proof}

Finally, we can prove our main result 
for Schatten classes, which is Theorem~\ref{thm.a2}. We state it again for convenience:

\begin{theorem}\label{thm.e3}
Let $p$, $q$, $r$ be finite positive indices satisfying $\frac1p=\frac1q+\frac1r$.
Let $G_0\in\Smooth_q(H_0)$ and $G_1\in\Smooth_r(H_1)$.
Then for any $f\in B_{p,p}^{1/p}(\bbR)\cap\BMO(\bbR)$, we have $D(f)\in\Sch_p$ and 
$$
\norm{D(f)}_{p}
\leq
(2\pi)
C_1(p)
A_{q,r}
\norm{f}_{B_{p,p}^{1/p}(\bbR)},
$$
where $C_1(p)$ is the constant from \eqref{d8} and $A_{q,r}$ is the constant from \eqref{c3a}. 
This extends to $q=\infty$ (resp. $r=\infty$), if the class $\Smooth_q(H_0)$ (resp. $\Smooth_r(H_1)$)
is replaced by $\Smooth(H_0)$ (resp. $\Smooth(H_1)$). 
\end{theorem}
\begin{proof}
By Theorem~\ref{thm.e2}, Theorem~\ref{thm.c6} and Lemma~\ref{lma.d4}, we have
$$
\norm{D(f)}_p
=
\norm{\DOI(\wf)}_p
\leq 
A_{q,r} \norm{\wf}_p
\leq 
A_{q,r} (2\pi) 
C_1(p)\norm{f}_{B_{p,p}^{1/p}(\bbR)}.
\qedhere
$$
\end{proof}

\section{Sharpness and some extensions}\label{sec.g}

This section contains some additional information. We demonstrate the 
sharpness of our main result and give some extensions.

\subsection{Sharpness of estimates}
Here we construct a pair of self-adjoint operators $H_0$, $H_1$ in $L^2(\bbR)$ 
such that the estimates from Theorems~\ref{thm.a1} and \ref{thm.a2}
are saturated. Thus, this construction demonstrates that these 
estimates are sharp. We construct $H_0$ and $H_1$ as follows.

Let $H_0$ be the multiplication operator $\calM$ in $L^2(\bbR)$ from \eqref{b6c}. 
Let $J$ be the Hilbert transform,
$$
Jf(x)=\frac1{\pi i}\text{p.v.}\int_{-\infty}^\infty \frac{f(y)}{y-x}dy,
\quad
f\in L^2(\bbR).
$$
It is well known that $J$ is unitary in $L^2(\bbR)$; it is also evident that $J^*=J$.
We set
$$
H_1=JH_0J, \quad \Dom H_1=\{Ju: u\in\Dom H_0\}.
$$
Next, we would like to represent the difference $H_1-H_0$ as a product $G_1^*G_0$. 
Let $\calK=\bbC$ and let  $G_0:L^2(\bbR)\to\bbC$ be as in Example~\ref{exa.b3}: 
$$
G_0u=\int_{-\infty}^\infty u(x)dx, \quad u\in\Dom H_0.
$$
The operator $G_0$ is not closable, but $G_0\in\Smooth(H_0)$, with $\norm{G_0}_{\Smooth(H_0)}=1$.
Further, we set 
$$
G_1u=\frac{1}{\pi i}G_0Ju, \quad \Dom G_1=\Dom H_1.
$$
Clearly, $G_1\in\Smooth(H_1)$ with $\norm{G_1}_{\Smooth(H_1)}=1/\pi$. 
Thus, the constant $A$ (see \eqref{c3a}) equals $A=1/\pi$ in this case. 
We have
\begin{theorem}\label{thm.sharp}
Let $H_0$, $H_1$, $G_0$, $G_1$ be as described above. 
Then: 
\begin{enumerate}[\rm (i)]
\item
The identity \eqref{e0} holds true (i.e. $H_1=H_0+G_1^*G_0$ in the sesquilinear form sense).
\item
For any $f\in\BMO(\bbR)$, we have
$$
f(H_1)-f(H_0)=\frac1{\pi i}\wf J.
$$
Thus, 
\begin{gather}
\norm{f(H_1)-f(H_0)}=\frac1\pi \norm{\wf}=2\norm{f}_{\BMO}
=(2\pi) A \norm{f}_{\BMO};
\label{e6}
\\
f\in B_{p,p}^{1/p}(\bbR)\cap\BMO(\bbR)\Leftrightarrow D(f)\in\Sch_p.
\notag
\end{gather}
\end{enumerate}
\end{theorem}
\begin{proof}
Let $u,w\in L^\infty_\comp(\bbR)$, and let $v=Jw$.
Consider the left side of \eqref{e0}:
\begin{multline*}
(u,H_1v)-(H_0u,v)
=
(u,JH_0w)-(H_0u,Jw)
=
(Ju,H_0w)-(JH_0u,w)
\\
=
\frac1{\pi i}\int_\bbR \int_\bbR \frac{u(y)(x-y)\overline{w(x)}}{y-x}dy\,dx
=
-\frac1{\pi i} \overline{G_0w} G_0u=\overline{G_1v}G_0u, 
\end{multline*}
which is the right side of \eqref{e0}.
Next, 
$$
f(H_1)-f(H_0)
=
Jf(H_0)J-f(H_0)
=
(Jf(H_0)-f(H_0)J)J
=
\frac1{\pi i}\wf J.
$$
From here we get the first identity in \eqref{e6}. 
The middle identity in \eqref{e6} follows from Lemma~\ref{lma.d2}, 
and the rest 
follows from Lemma~\ref{lma.d4}. 
\end{proof}

\subsection{Quasicommutators}
Let $H_0$ and $H_1$ be self-adjoint operators in $\calH$, and let $J$ be a bounded operator in $\calH$. 
Here we consider the so-called quasicommutators
\[
D_J(f):=f(H_1)J-Jf(H_0).
\label{g1}
\]
Let us assume that 
\[
H_1J-JH_0=G_1^*G_0
\label{g2}
\]
with some operators $G_0$, $G_1$ acting from $\calH$ to $\calK$ such that
\[
G_0\in\Smooth(H_0)\quad \text{ and }\quad G_1\in\Smooth(H_1).
\label{g2a}
\] 
As usual, \eqref{g2} should be understood in the sesquilinear form sense, i.e.
\[
(Ju,H_1v)-(H_0u,J^*v)=(G_0u,G_1v), \quad u\in\Dom(H_0), \quad v\in\Dom(H_1).
\label{g3}
\]
The resolvent identity in this case takes the form
\[
R_1(z)J-JR_0(z)=-(G_1R_1(\overline{z}))^*G_0R_0(z)=-(G_0R_0(\overline{z}))^*G_1R_1(z).
\label{g3b}
\]
Similarly to \eqref{e0a}, we define the sesquilinear form
$$
d_{J,f}[u,v]:=(Ju,\overline{f}(H_1)v)-(f(H_0)u,J^*v),
\quad
u\in\Dom f(H_0) , \quad v\in\Dom f(H_1). 
$$
For bounded functions $f$ the quasicommutator $D_J(f)$ can be defined directly as in \eqref{g1} and 
\[
d_{J,f}[u,v]=(D_J(f)u,v), \quad  u\in\Dom f(H_0), \quad v\in\Dom f(H_1). 
\label{g3a}
\]
Similarly to Proposition~\ref{prp.e1}, we have
\begin{proposition}\label{prp.g2}
Assume \eqref{g2a} and \eqref{g3}. Then we have 
$d_{J,f}[u,v]=0$
if $u\in\calH^{(\sing)}(H_0)\cap \Dom(f(H_0))$ or $v\in\calH^{(\sing)}(H_1)\cap \Dom(f(H_1))$
(or both). 
\end{proposition}
\begin{proof}
If $u\in\calH^{(\sing)}(H_0)$, then for all $\Im z\not=0$ we have $G_0R_0(z)u=0$ and so, 
by the resolvent identity \eqref{g3b}, 
$$
R_1(z)Ju=JR_0(z)u. 
$$
From here, as in the proof of Proposition~\ref{prp.e1}, we obtain
$d_{J,f}[u,v]=0$ for any $f$ such that 
$u\in \Dom f(H_0)$ and $v\in \Dom f(H_1)$. 
The case $v\in\calH^{(\sing)}(H_1)$ is considered in the same way. 
\end{proof}
In full analogy with Theorem~\ref{thm.e1}, we have

\begin{theorem}\label{thm.g1}
Assume  \eqref{g2a} and \eqref{g3}. 
For any $f\in\BMO(\bbR)$ and for all $u\in\calH^{\ac}(H_0)$, $v\in\calH^{\ac}(H_1)$, 
the sesquilinear form $d_{J,f}$ satisfies the bound
$$
\abs{d_{J,f}[u,v]}\leq 2\pi A\norm{f}_{\BMO}\norm{u}_\calH\norm{v}_\calH, 
$$
where $A$ is the constant \eqref{c3a}.
Thus, the form $d_{J,f}[u,v]$ corresponds to a bounded operator $D_J(f)$ on $\calH$
in the sense of \eqref{g3a}, 
and $D_J(f)$ satisfies the norm bound 
$$
\norm{D_J(f)}\leq 2\pi A\norm{f}_{\BMO(\bbR)}. 
$$
If $f_n\to f$ $*$-weakly in $\BMO(\bbR)$, then $D_J(f_n)\to D_J(f)$ $*$-weakly in $\calB(\calH)$.
\end{theorem}
The proof repeats the proof of Theorem~\ref{thm.e1} word for word;
the only difference is that the required resolvent identity in this case has the form \eqref{g3b}. 

Furthermore, repeating word for word the proof of Theorem~\ref{thm.e2}, we establish the 
modified Birman-Solomyak formula 
$$
D_J(f)=\DOI(\wf)
$$
for all $f\in \BMO(\bbR)$. 
Thus, we can apply the compactness Lemma~\ref{lma.c5} and the Schatten bounds Theorem~\ref{thm.c6}:
\begin{theorem}\label{thm.g2}
Assume \eqref{g2a} and \eqref{g3}; 
let $D_J(f)$ be as defined above. Assume $f\in \CMO(\bbR)$ and assume in addition that at least
one of the inclusions 
$$
G_0\in\Smooth_\infty(H_0), \quad G_1\in\Smooth_\infty(H_1)
$$ 
holds true. 
Then $D_J(f)$ is compact. 
Further, let $p$, $q$, $r$ be finite positive indices satisfying $\frac1p=\frac1q+\frac1r$, 
and let $A_{q,r}$ be as in \eqref{c3a}. Then the Schatten class bound
$$
\norm{D_J(f)}_p\leq 2\pi C_1(p)A_{q,r}\norm{f}_{B_{p,p}^{1/p}(\bbR)}, 
$$
holds true for all $f\in B_{p,p}^{1/p}(\bbR)\cap\BMO(\bbR)$.
It extends to $q=\infty$ (resp. $r=\infty$), if one replaces the class $\Smooth_q(H_0)$ (resp. $\Smooth_r(H_1)$)
by $\Smooth(H_0)$ (resp. $\Smooth(H_1)$). 
\end{theorem}

\subsection{Products of functions}\label{sec.g7}

Let $H_0$ and $H_1$ be self-adjoint operators in $\calH$, and let $\varphi_0, \varphi_1\in L^\infty(\bbR)$. 
Here we consider the products
\[
\varphi_1(H_1)^*D(f)\varphi_0(H_0),
\label{g4a}
\]
where $D(f)=f(H_1)-f(H_0)$ as before. 
The main interest of this is in taking $\varphi_0=\varphi_1=\1_\Lambda$,
where $\Lambda\subset\bbR$; this leads to \emph{local} variants
of smoothness conditions. We develop this in more detail in the forthcoming publication \cite{II}.

We assume that 
\[
H_1-H_0=G_1^*G_0
\label{g5}
\]
for some $G_0,G_1:\calH\to\calK$, where $G_0$ is $H_0$-bounded and $G_1$ is $H_1$-bounded. 
As usual, \eqref{g5} should be understood in the sesquilinear form sense, see \eqref{e0}. 
Our smoothness assumptions are now as follows:
\[
G_0\varphi_0(H_0)\in\Smooth(H_0), \quad
G_1\varphi_1(H_1)\in\Smooth(H_1).
\label{g6}
\]
We define the operator \eqref{g4a} via the sesquilinear form
$$
d[u,v]:=(\varphi_0(H_0)u,\overline{f}(H_1)\varphi_1(H_1)v)-(f(H_0)\varphi_0(H_0)u,\varphi_1(H_1)v),
$$
for $u\in\Dom f(H_0)$ and $v\in\Dom f(H_1) $.

\begin{theorem}
Assume \eqref{g5} and \eqref{g6}; let $f\in\BMO(\bbR)$ and let $d$ be as above. 
Then $d[u,v]=0$, if  $u\in\calH^{(\sing)}(H_0)\cap \Dom f(H_0)$ or $v\in \calH^{(\sing)(H_1)}\cap\Dom f(H_1)$. 
Further, 
for $u\in\calH^{\ac}(H_0)$ and $v\in\calH^{\ac}(H_1)$, 
the sesquilinear form $d$ satisfies the bound
$$
\abs{d[u,v]}\leq
2\pi \norm{f}_{\BMO(\bbR)}
\norm{G_0\varphi_0(H_0)}_{\Smooth(H_0)}
\norm{G_1\varphi_1(H_1)}_{\Smooth(H_1)}
\norm{u}_\calH\norm{v}_\calH. 
$$
Thus, the sesquilinear form $d$ corresponds to a bounded operator 
$\varphi_1(H_1)^*D(f)\varphi_0(H_0)$, which satisfies 
$$
\norm{\varphi_1(H_1)^*D(f)\varphi_0(H_0)}
\leq
2\pi \norm{f}_{\BMO(\bbR)}
\norm{G_0\varphi_0(H_0)}_{\Smooth(H_0)}
\norm{G_1\varphi_1(H_1)}_{\Smooth(H_1)}.
$$
If $f_n\to f$ $*$-weakly in $\BMO(\bbR)$, then 
$$
\varphi_1(H_1)^*D(f_n)\varphi_0(H_0)\to \varphi_1(H_1)^*D(f)\varphi_0(H_0)
$$
$*$-weakly in $\calB(\calH)$. 
\end{theorem}
\begin{proof}
Let 
$$
J=\varphi_1(H_1)^*\varphi_0(H_0)
$$
and let $d_{J,f}$ be as defined in \eqref{g3a}.
Observe that we have
$$
H_1J-JH_0=(G_1\varphi_1(H_1))^*(G_0\varphi_0(H_0))
$$
in the sesquilinear form sense, and 
$$
\varphi_1(H_1)^*(f(H_1)-f(H_0))\varphi_0(H_0)
=
f(H_1)\varphi_1(H_1)^*\varphi_0(H_0)
-
\varphi_1(H_1)^*\varphi_0(H_0)f(H_0),
$$
or, in different notation, 
$$
d[u,v]=d_{J,f}[u,v]. 
$$
Thus, the operator identity
\[
\varphi_1(H_1)^*D(f)\varphi_0(H_0)=D_J(f)
\label{g8}
\]
holds true and our claims follow immediately from Proposition~\ref{prp.g2} and Theorem~\ref{thm.g1}.
\end{proof}
As an immediate consequence of \eqref{g8} and of Theorem~\ref{thm.g2}, 
we also obtain the corresponding compactness result and the Schatten norm bounds. 
\begin{theorem}
Assume \eqref{g5} and \eqref{g6}, 
and let $f\in \CMO(\bbR)$. 
Assume that at least one of the two inclusions
$$
G_0\varphi_0(H_0)\in\Smooth_\infty(H_0),\quad  
G_1\varphi_1(H_1)\in\Smooth_\infty(H_1)
$$ 
holds true. 
Then $\varphi_1(H_1)^*D(f)\varphi_0(H_0)$ is compact. 
Further, let $p$, $q$, $r$ be finite positive indices such that $\frac1p=\frac1q+\frac1r$,
and let $f\in B_{p,p}^{1/p}(\bbR)\cap\BMO(\bbR)$. 
Then we have the bounds
\begin{multline*}
\norm{\varphi_1(H_1)^*D(f)\varphi_0(H_0)}_p
\\
\leq
2\pi C_1(p)\norm{f}_{B_{p,p}^{1/p}(\bbR)}
\norm{G_0\varphi_0(H_0)}_{\Smooth_q(H_0)}
\norm{G_1\varphi_1(H_1)}_{\Smooth_r(H_1)}. 
\end{multline*}
This extends to the case $q=\infty$ (resp. $r=\infty$), if one replaces
the class $\Smooth_q(H_0)$ (resp. $\Smooth_r(H_1)$) by 
$\Smooth(H_0)$ (resp. $\Smooth(H_1)$). 
\end{theorem}

\appendix
\section{Two technical proofs}

\begin{proof}[Sketch of proof of Proposition~\ref{prp.a1}]
The key point is the calculation of the asymptotics of the Fourier transform of $F_\alpha$. 
A lengthy but straightforward calculation (see e.g. \cite[Section 4]{PYa}) yields that for $a_+\not=a_-$ we have
\[
\wh F_\alpha(t)=\frac{a_+-a_-}{2\pi i} \frac1t (\log \abs{t})^{-\alpha}+O(t^{-1}(\log \abs{t})^{-\alpha-1}), \quad t\to\pm\infty, 
\label{A2}
\]
and for $a_+=a_-$ we have
\[
\wh F_\alpha(t)=a_+\alpha\frac1{t}(\log \abs{t})^{-\alpha-1}+O(t^{-1}(\log \abs{t})^{-\alpha-2}), \quad t\to\pm\infty.
\label{A3}
\]
In both cases, the $O(\cdot)$ terms can be differentiated arbitrary many times, i.e. 
\[
(d/dt)^m O(t^{-1}(\log \abs{t})^{-\alpha-1})=O(t^{-1-m}(\log \abs{t})^{-\alpha-1}). 
\label{A4}
\]
First consider the case $a_+\not=a_-$ and let us check that the series \eqref{A1} converges
if and only if $\alpha>1/p$. It is easy to see that $\wh F_\alpha(t)$ is a $C^\infty$-smooth function of $t\in\bbR$
and as a consequence, the series over $j\leq0$ converges for all $p>0$ and $\alpha\in\bbR$. 
Thus, it suffices to inspect the convergence of the series over $j\geq0$. 
By the asymptotics \eqref{A2}, we have 
\begin{multline*}
(F_\alpha*\wh w_j)(x)=\int_0^\infty e^{ixt} \wh F_\alpha(t) w_j(t)dt
=
\frac{a_+-a_-}{2\pi i}
\int_0^\infty \frac1t(\log t)^{-\alpha}w(t/2^j)e^{ixt}dt
+\text{error}
\\
=
\frac{a_+-a_-}{2\pi i}
\int_0^\infty \frac1t(\log 2^jt)^{-\alpha}w(t)e^{i2^jxt}dt
+\text{error}
\\
=
\frac{a_+-a_-}{2\pi i}(\log 2^j)^{-\alpha}
\int_0^\infty \frac1t\bigl(1+\frac{\log t}{\log 2^{j}}\bigr)^{-\alpha}w(t)e^{i2^jxt}dt
+\text{error}
\\
=
\frac{a_+-a_-}{2\pi i}(\log 2^j)^{-\alpha}\varphi(2^jx)
+\text{error}
\end{multline*}
where $\varphi$ is a Schwartz class function and the error term can be 
controlled by using the estimates \eqref{A4}. It follows that 
$$
2^j\norm{F_\alpha*\wh w_j}_{L^p(\bbR)}^p=C j^{-\alpha p}+o(j^{-\alpha p}), \quad j\to\infty,
$$
where $C\not=0$. In the same way we get
$$
2^j\norm{F_\alpha*\overline{\wh w_j}}_{L^p(\bbR)}^p=C j^{-\alpha p}+o(j^{-\alpha p}), \quad j\to\infty.
$$
It follows that the series in \eqref{A1} for $f=F_\alpha$ converges if and only if $p\alpha>1$. 

In the same way, considering the case $a_+=a_-$ and using the asymptotics \eqref{A3}, we 
conclude that the series \eqref{A1} converges if and only if $p(\alpha+1)>1$. 
\end{proof}

Finally, we give the

\begin{proof}[Proof of Lemma~\ref{lma.d1}]
The proof is effected through mapping the problem to the unit circle. 

\emph{Step 1:}
First we need to consider the analogous problem in the space $\BMO(\bbT)$, which is defined as follows. 
For $h\in L^2(\bbT)$ and $v\in H^1(\bbD)$ ($=$ the standard Hardy class on the unit disk), let 
$$
t_h(v):=\lim_{r\to1-}\int_{-\pi}^\pi h(e^{i\theta})v(re^{i\theta})\frac{d\theta}{2\pi},
$$
if the limit exists. Then $h\in\BMO(\bbT)$ if and only if both linear functionals 
$t_h$ and $t_{\overline{h}}$ are bounded on $H^1(\bbD)$. 

Let us prove that for $h\in \BMO(\bbT)$, its approximations by Fejer sums 
converge to $h$ $*$-weakly in $\BMO$. More precisely, set
\begin{equation}
h_n(z)=\sum_{j=-n}^n \biggl(1-\frac{\abs{j}}{n}\biggr)\hat h_j z^j, 
\quad\text{ where }\quad
h(z)=\sum_{j=-\infty}^\infty \hat h_j z^j. 
\label{d4b1}
\end{equation}
It is easy to see that the linear map $h\mapsto h_n$ is bounded on $H^1(\bbD)$: 
$$
\norm{h_n}_{H^1(\bbD)}\leq C\norm{h}_{H^1(\bbD)}, 
$$
and so, by duality,
\begin{equation}
\norm{h_n}_{\BMO(\bbD)}\leq C\norm{h}_{\BMO(\bbD)}. 
\label{d4b2}
\end{equation}
Next, it is clear that if $v\in H^1(\bbD)$ is a trigonometric polynomial, then 
\begin{equation}
t_{h_n}(v)\to t_{h}(v), \quad n\to\infty.
\label{d4b3}
\end{equation}
Since trigonometric polynomials are dense in $H^1(\bbD)$, by an approximation 
argument (involving \eqref{d4b2}), we obtain \eqref{d4b3} for all $v\in H^1(\bbD)$. 
Similarly, one proves that $t_{\overline{h}_n}(v)\to t_{\overline{h}}(v)$. 

\emph{Step 2:}
Let $\omega$ be the standard conformal map from the unit disk to the upper half-plane:
$$
\omega:\bbD\to\bbC_+,
\quad
\omega(\zeta)=i\frac{1+\zeta}{1-\zeta}, \quad \zeta\in\bbD.
$$
Recall (see e.g. \cite[Cor. VI.1.3]{Garnett}) 
that $f\in\BMO(\bbR)$ if and only if
$h=f\circ\omega\in\BMO(\bbT)$. 
Let $h_n$ be the Fejer sum \eqref{d4b1} of $h$, and let 
$f_n=h_n \circ \omega^{-1}$. 
By construction, $f_n$ is a rational function; let us prove that 
$f_n\to f$ $*$-weakly in $\BMO(\bbR)$. 
For $u\in H^1(\bbC_+)$, let $v\in H^1(\bbD)$ be given by 
$$
v(\zeta)=-{4\pi}(1-\zeta)^{-2} u(\omega(\zeta)).
$$
Then a direct calculation shows that 
$$
T_f(u)=t_h(\zeta v), \text{ and } T_{f_n}(u)=t_{h_n}(\zeta v).
$$
Now we get $T_{f_n}(u)\to T_{f}(u)$ by the first step of the proof. 
Similarly, one proves $T_{\overline{f}_n}(u)\to T_{\overline{f}}(u)$.
\end{proof}


\end{document}